%% file: main.tex
\documentclass{article}

\usepackage{arxiv}

\usepackage{hyperref}

\usepackage{babel}
\usepackage[]{geometry}
\DeclareMathAlphabet\mathbfcal{OMS}{cmsy}{b}{n}

\usepackage{amsmath,amsthm,amssymb,babel,amsfonts,graphics}

\usepackage{caption}
\usepackage{subcaption}

\usepackage{amsmath,amssymb,amsfonts,latexsym,cancel}
\usepackage{rawfonts}
\usepackage{pictexwd}
\usepackage{tikz}

\usepackage{tikz-cd}

\usepackage{pgfplots}

\pgfplotsset{compat=newest}
\pgfplotsset{plot coordinates/math parser=false}
\newlength\figureheight
\newlength\figurewidth

\usepackage{color}
\usepackage{epsfig}

\newcommand{\R}{\mathbb{R}}
\newcommand{\N}{\mathbb{N}}

\def\1{\raisebox{2pt}{\rm{$\chi$}}}

\newcommand{\bfr}{\mathbf{r}}
\newcommand{\bfn}{\mathbf{n}}
\newcommand{\bfx}{\mathbf{x}}

\theoremstyle{plain}

\newtheorem{proposition}{Proposition}

\newtheorem{theorem}{Theorem}
\newtheorem{corollary}{Corollary}
\newtheorem{lemma}{Lemma}
\newtheorem{remark}{Remark}
\newtheorem{assumption}{Assumption}

\theoremstyle{definition}

\theoremstyle{remark}

\numberwithin{equation}{section}
\newcommand{\obs}{\Omega} % Obstacle(s)
\newcommand{\free}{\mathcal{D}} % Free domain
\newcommand{\target}{\mathcal{T}} % Target set
\newcommand{\game}{\mathcal{G}} % Game domain
\newcommand{\boundary}{\partial\game} %Boundary of game
\newcommand{\boundaryT}{\partial\target} %Boundary of the target set
\newcommand{\usable}{\partial\target^\ast} % Usable part of the boundary
\newcommand{\levelset}{\mathcal{L}} % Level set
\newcommand{\front}{\Gamma} % Front

\newcommand{\hu}{\mathbf{x}_P^+} % Upper-horizon for visib
\newcommand{\hl}{\mathbf{x}_E^-} % Lower-horizon for visib

\newcommand{\hus}{\mathbf{s}_P^+}
\newcommand{\hls}{\mathbf{s}_E^-}

\usepackage[normalem]{ulem} % for del command

\begin{document}

\title{Usable boundary for visibility-based surveillance-evasion games}

\author{ 
 Carlos Esteve-Yag\"ue
 \\
	Department of Applied Mathematics \\
	and Theoretical Physics,
	University of Cambridge\\
	Cambridge, UK \\
	\texttt{ce423@cam.ac.uk} \\
    \And
 Richard Tsai
 \\
	Department of Mathematics and Oden Institute \\
	for Computational Engineering and Sciences\\
	The University of Texas at Austin\\
	Austin, USA \\
	\texttt{ytsai@math.utexas.edu} \\
	%% \AND
	%% Coauthor \\
	%% Affiliation \\
	%% Address \\
	%% \texttt{email} \\
	%% \And
	%% Coauthor \\
	%% Affiliation \\
	%% Address \\
	%% \texttt{email} \\
	%% \And
	%% Coauthor \\
	%% Affiliation \\
	%% Address \\
	%% \texttt{email} \\
}

\hypersetup{
% pdftitle={A template for the arxiv style},
% pdfsubject={q-bio.NC, q-bio.QM},
pdfauthor={Carlos Esteve-Yag\"ue, Richard Tsai},
pdfkeywords={differential games, pursuit-evasion games, discontinuous viscosity solutions, usable part of the boundary, semi-permeable barriers},
}

\date{\today}

\maketitle

\begin{abstract}
We consider a surveillance-evasion game in an environment with obstacles. In such an environment, a mobile pursuer seeks to maintain the visibility with a mobile evader, who tries to get occluded from the pursuer in the shortest time possible. In this two-player zero-sum game setting, we study  the discontinuities of the value of the game near the boundary of the target set (the non-visibility region). In particular, we describe the transition between the usable part of the boundary of the target (where the value vanishes) and the non-usable part (where the value is positive). We show that the value enjoys a different behaviour depending on the regularity of the obstacles involved in the game. Namely, we prove that the boundary profile is continuous for the case of smooth obstacles, and that it exhibits a jump discontinuity when the obstacle contains corners. Moreover, we prove that, in the latter case, there is a semi-permeable barrier emanating from the interface between the usable and the non-usable part of the boundary of the target set.
\end{abstract}



%Sections
\input{sections/intro}
\input{sections/main_results}

\input{sections/preliminaries}

\input{sections/discontinuities}
\input{sections/boundary-smooth}
\input{sections/boundary-corner}
\input{sections/conclusion}

\section*{Acknowledgment}
Tsai is partially supported by Army Research Office, under Cooperative Agreement Number W911NF-19-2-0333 and National Science Foundation Grant DMS-2110895.

%Appendix
\appendix
\input{sections/appendix_2}

\bibliographystyle{abbrv}
\bibliography{mybibfile}

\end{document}

%% file: sections/intro.tex
\section{Introduction}

We study the surveillance problem in which 
a group of mobile \textit{pursuers} (or observers) seek to maintain the line-of-sight with a group of mobile targets in an environment with obstacles, which constraint the pursuers' visibility and the mobility of both pursuers and targets. 
Surveillance in complex dynamic situations arises in many applications, such as crime prevention, wildlife research, sport coverage, traffic monitoring and industrial processes.
A special class of surveillance problem is the so-called target tracking, in which the observers need to maintain the visibility over the targets without the knowledge of their future movements.
In order to plan an optimal surveillance strategy while lacking such an important piece of information, one often assumes that either the targets move randomly \cite{frew2003trajectory,zhou2008optimal}, or that they behave in an adversarial manner \cite{bhattacharya2016visibility, bhattacharya2008approximation, Takei_Tsai_Zhou_Landa_2014}, trying to evade the pursuers' visibility.
The former assumption would lead to an optimal control problem, whereas the latter one brings us into the framework of game theory.

In this work we consider the latter case, in which the targets are adversarial in nature, so from now on, we will refer to them as \textit{the evaders}.
More precisely, we consider the surveillance-evasion game in which, given the initial position of the players (pursuers and evaders), the evaders try to minimise the time to occlusion, whilst the pursuers try to maintain the line-of-sight with the evaders for as long as possible.
In this setup, according to Isaacs' nomenclature \cite{isaacs1965differential},
we have a \emph{game of degree}. This stands in contrast with the related \emph{game of kind}, where, based on the initial positions of the players, the sole concern is to determine if the evaders can successfully hide from the pursuers within a finite time.

In the game of degree, the object of interest is the so-called \textit{value of the game}, which for every player, represents the optimal output of the game (time to occlusion), given the initial position, and assuming optimal play by the opponent. 
In a two-player zero-sum game, such as the one we are considering here, the value can be defined from the perspective of the evaders (lower value) or the pursuers (upper value), and the game is said to \textit{have value} when both values coincide.  See section \ref{sec: preliminaries} for further details about the definition and existence of the value.
The value of the game can also be used to characterise the victory domains of the associated game of kind.
Namely, the victory domain for the evaders is the set of initial positions for which the value is finite, whereas the victory domain for the  pursuers is the set where the value is infinity, as they can delay the occlusion time forever.

Assuming certain control models for the dynamics of the players and optimisation in a zero-sum game setup, one can use a dynamic programming approach to derive the partial differential equation for the value of the game,  known as the Hamilton-Jacobi-Isaacs HJI equation (see \cite{bhattacharya2008approximation, 
cardaliaguet1999set,cardaliaguet2000pursuit, evans1984differential, Takei_Tsai_Zhou_Landa_2014} and the references therein).
This is the analogue of the Hamilton-Jacobi-Bellman equation arising in optimal control theory \cite{bardi1997optimal}.
In the framework of differential game theory, 
this partial differential equation was first introduced  by Isaacs in \cite{isaacs1965differential} and gives a characterisation of the value at regions where it is sufficiently smooth.
More interestingly,  at the points where the value is differentiable,  its gradient can be used to obtain an optimal feedback control for the players.
However, the value fails to be smooth in general, and even worse, it is well-known that in differential game problems, the value may develop discontinuities.
Since the initial works by Isaacs in the 1960's,  there has been a great interest in studying the structure and properties of the discontinuity set of the value function.
Understanding the set of discontinuities is of major importance when implementing numerical methods to approximate the discontinuous viscosity solution to the HJI equation.
Indeed, most numerical schemes provide provable good approximations of the value only in compact sets away from the discontinuities \cite{bardi1995convergence, Bardi1999, falcone2006numerical}.

In the particular case of the surveillance-evasion game that we consider here,
the structure  of the discontinuity set is known to be rather complex even in the simplest situations.
These discontinuities can be different in nature. For instance, a special case of discontinuity is that occurring on the boundary separating the victory domain for the evaders (where the value is finite) and the victory domain for the pursuers (where the value is infinite).
Another special case of discontinuity is related to the \textit{non-usable part} of the boundary of the target set (or non-visibility set\footnote{By non-visibility set, we refer to the positions of the game in which all the evaders are occluded from the pursuers. Since the non-visibility set determines the end of the game, we call it the target set indistinctly.}). The non-usable part consists of the positions of the game which are on the boundary of the target set, but are unreachable assuming that the pursuers play optimally.
On the contrary, the \textit{usable part} of the boundary consists of those positions of the game, on the boundary of the target set, 
that can actually be used by the evaders to end the game by getting occluded from the pursuers.

In order to illustrate the notion of usable and non-usable part of the boundary of the target set in our surveillance-evasion setting, let us consider the two-dimensional case with only two players (one evader and one pursuer).
Let $\obs\subset\R^2$ be a given open set with Lipschitz boundary representing the obstacle (or obstacles if $\obs$ has multiple connected components).
The free-domain in which the players can move is defined as $\free := \R^2\setminus \obs$, and any position of the game is denoted by a couple $(E,P)\in \free^2$ representing the positions of the evader and the pursuer respectively.
The target set $\target\subset \free^2$ consists of the points $(E,P)\in \free^2$ such that the line segment joining $E$ and $P$, denoted by $[E,P]$, intersects the obstacle $\obs$, i.e.
$$
\target := \{ (E,P)\in \free^2 \ : \quad [E,P]\cap \obs \neq \emptyset \}.
$$ 
We denote by $\boundaryT$ the boundary of the target set.
Here, the boundary is taken with respect to the topology relative to the set $\free^2$, i.e., the topology generated by the open balls in $\R^4$ intersected with $\free^2$.
In this way, $\boundaryT$ contains the points $(E,P)$ such that the line segment $[E,P]$ is tangent to $\obs$, but does not contain the pairs $(E,P)$ such that $E$ or $P$ lies on $\partial\obs$. See Figure \ref{fig:partialD} for an illustration of the boundary of the target set.
Given the initial position of the game $(E,P)\in \free^2 \setminus \target$, the end-game time is defined as the first time such that the position reaches $\boundaryT$.

\begin{figure}
    \centering
    \includegraphics[scale = .5]{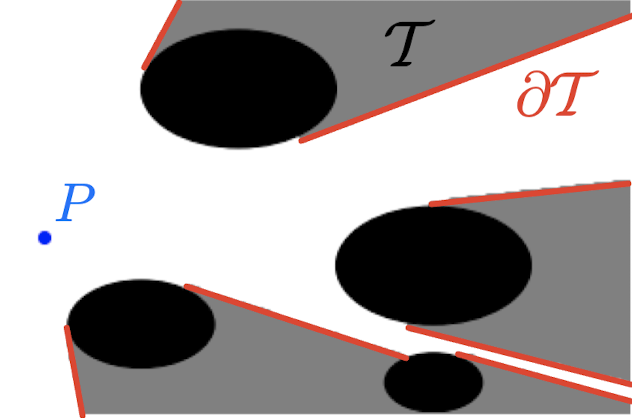}
    \hspace{1cm}
    \includegraphics[scale = .5]{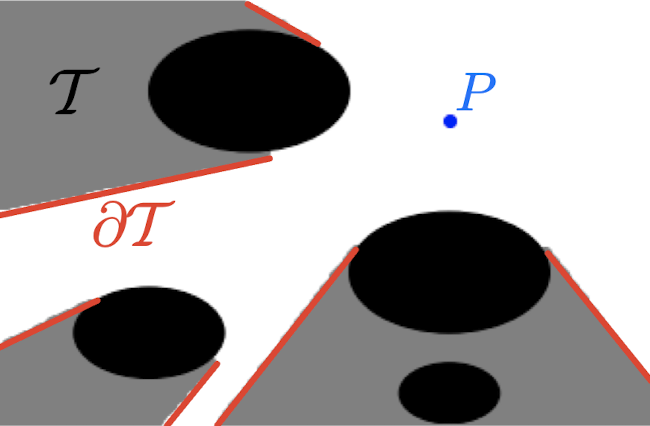}
    \caption{Representation (in red) of the boundary of the target set $\boundaryT$ from the perspective of the evader. The pictures represent two slices of the four-dimensional domain in which the blue point representing the pursuer's position is fixed. Note that the boundary is taken with respect to the topology restricted to the free domain $\free^2 = [\R^2 \setminus \obs]^2$, so the part of the boundary of the obstacles which is visible from the pursuer's position do not belong to $\boundaryT$.}
    \label{fig:partialD}
\end{figure}

Assuming that the dynamics for $E$ and $P$ are given by
$$
\dot{E}(t) = f(E(t), a(t))
\quad
\text{and}
\quad
\dot{P}(t) = g(P(t), b(t)),
\quad \text{for} \ t>0,
$$
where $a(\cdot): (0,\infty) \to \mathcal{A}$ and $b(\cdot):(0,\infty)\to \mathcal{B}$ are measurable functions representing the controls, with $\mathcal{A}$ and $\mathcal{B}$ being the compact control sets,
we define the Hamiltonian of the game as
$$
H(E,P, \rho_E, \rho_P):= \max_{a \in \mathcal{A}} - f(E, a)\cdot\rho_E
 + \min_{b\in \mathcal{B}} - g(P,b)\cdot \rho_P.
$$

One can prove (see Corollary \ref{cor: usable part general}) that  any point $(E,P)\in \boundaryT$ satisfying
$H (E,P, \bfn_E,\bfn_P) >0,$
where $(\bfn_E, \bfn_P)$ is the outer normal vector to $\target$ at $(E,P)$, is in the usable part of the boundary, denoted by $\usable$.
If on the contrary it holds that 
$H (E,P, \bfn_E,\bfn_P)<0,$ 
then $(E,P)$ is in the non-usable part of the boundary, denoted by $\boundaryT\setminus\usable$. See section \ref{sec: discontinuities} for further details.
What happens on the boundary between these two regimes (i.e. when $H(E,P,\bfn_E,\bfn_P)=0$) is more intricate, and is the main object of study of this work. As we will see,  the value of the game exhibits different behaviours near these points, depending on the regularity of the obstacles.

At the level of the dynamics, the simplest case that one can consider is that of homogeneous and isotropic dynamics, i.e. when the players can move in any direction at a maximum speed that does not depend on the position. 
In other words, when the evolution of the game is given by
\begin{equation}
\label{homogeneous isotropic intro}
\dot{E}(t) = \gamma_e a(t)
\quad
\text{and}
\quad
\dot{P}(t) = \gamma_p b(t),
\quad \text{for} \ t>0,
\end{equation}
where $\gamma_e,\gamma_p>0$ are two given constants representing the maximum speed for the evader and the pursuer respectively, and the control sets $\mathcal{A}$ and $\mathcal{B}$ are simply the unit ball in $\R^2$.
The Hamiltonian associated to these dynamics is independent of the position of the game, and reads as
$$
H(\rho_E,  \rho_P) = \gamma_e | \rho_E| - \gamma_p | \rho_P|.
$$

In this case, for any boundary point $(E,P)\in \boundaryT$ such that the segment $[E,P]$ is tangent to the obstacle at a single point $\bfx^\ast\in \partial\obs$, one can easily determine whether or not $(E,P)$ is in the usable part of $\boundaryT$ in terms of the ratio between the speed of the players and their distance to the tangent point $\bfx^\ast$.
Namely, it holds that
$$
(E,P) \ \text{is in the usable part whenever} \ \dfrac{\gamma_e}{d_E} > \dfrac{\gamma_p}{d_P}
$$
and
$$
(E,P) \ \text{is in the non-usable part whenever} \ \dfrac{\gamma_e}{d_E} < \dfrac{\gamma_p}{d_P},
$$
where $d_E := |E-\bfx^\ast|$ and $d_P := |P-\bfx^\ast|$. See Figure \ref{fig: usable part} for an illustration of the usable part of $\boundaryT$.

\begin{figure}[t]
\centering
\includegraphics[scale=.5]{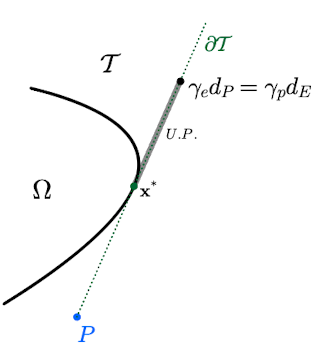}
\caption{Illustration of the usable part (U.P.)  of the boundary of the target set from the perspective of the evader in the case of isotropic, homogeneous dynamics.The figure represent a two-dimensional slice of the four-dimensional domain in which the position of the pursuer is fixed. The dotted green line above $\bfx^\ast$ represents the boundary of $\target$, and the shadow part represents the usable part of $\boundaryT$.}
\label{fig: usable part}
\end{figure}

In this work, we study the behaviour of the value of the game near the limits of the usable part of $\boundaryT$, i.e., positions of the game $(E,P)\in \free^2$ such that the line segment $[E,P]$ is tangent to the obstacle $\obs$ at a single point $\bfx^\ast$ and satisfy
\begin{equation}
\label{limit UP cond}
\gamma_e d_P= \gamma_p d_E.
\end{equation}
Namely, we provide estimates of the profile of the value of the game on the boundary of the target set, describing how it transitions from being zero on the usable part to positive on the non-usable part.
We also discuss the possibility of the existence of a semi-permeable barrier emanating from the limit of the usable part of $\boundaryT$.
This possibility was already discussed in \cite{bhattacharya2016visibility}, providing a negative answer for the case of a circular obstacle. Here,  we prove that the answer is not always negative, and actually depends on the regularity of the obstacle at $\bfx^\ast$.

In the next section we present the main results of the present paper concerning the behaviour of the value of the game near the boundary of the target set.
The rest of the paper is structured as follows:
in Section \ref{sec: preliminaries} we provide some necessary notation and a detailed description of the game in a general setting.
We also give the precise definition of the value and formulate the corresponding boundary value problem associated to the HJI equation.
In Section \ref{sec: discontinuities}, we use the level-set method to describe the value by means of the evolution of a propagating front.
This construction is then used to describe the discontinuities and to characterise the usable part of the boundary by means of the Hamiltonian.
In Section \ref{sec: boundary estimates smooth}, we carry out the analysis of the value near the boundary of the target set, for the case when the boundary of the obstacle is smooth.
In Section \ref{sec: boundary estimates corner}, we consider the case of boundary points for which the line-of-sight between the evader and the pursuers is tangent to the obstacle in a corner.
Finally, in Section \ref{sec: conclusions}, we discuss the conclusions of the present paper and possible future directions.
We also include, in Appendix \ref{sec: lemmas boundary}, two technical results which are used in the proofs presented in Section \ref{sec: boundary estimates smooth}.

%% file: sections/main_results.tex
\section{Main contributions}
\label{sec: main results}

Throughout this section,  we consider the surveillance-evasion game in a two-dimensional environment with obstacles and two players (one evader and one pursuer), moving according to the homogeneous and isotropic controlled dynamics defined in \eqref{homogeneous isotropic intro}.
Let us consider a point on the boundary of the target set
$(E,P)\in \boundaryT$ such that the segment $[E,P]$ is tangent to $\obs$ at a single point $\bfx^\ast\in \partial\obs$, and let us assume that \eqref{limit UP cond} holds.
Under this assumption, $(E,P)$ lies on the interface between the usable and the non-usable part of $\boundaryT$.
As we anticipated in the introduction, the value of the game near such boundary points, on the limits of the usable part of the boundary, exhibits a different behaviour depending on the regularity of $\partial\obs$ at $\bfx^\ast$.
Namely, we consider the two following cases: when $\partial\obs$ is smooth on a neighbourhood of $\bfx^\ast$ (see Figure \ref{fig:boundary points two cases} left); 
and when $\bfx^\ast$ is a corner of the obstacle (see Figure \ref{fig:boundary points two cases} right).
As we will see, the different behaviour does not only concern the boundary profile of the value but also the structure of the discontinuity set.

Let us start by describing the behaviour of the value in the first case, in which $\partial\obs$ is strictly convex and smooth in a neighbourhood of $\bfx^\ast$. More precisely, we assume that the curvature function $\kappa: \partial \obs\to \R^+$ satisfies
\begin{equation}
\label{assum boundary smooth intro}
\begin{cases}
\exists r>0, \ \exists \kappa_0>0 \ \text{and} \ \exists \lambda_r >0 
\quad \text{such that} \\
\kappa (\bfx) \geq \kappa_0 \quad \text{and}  \quad | \kappa (\bfx) - \kappa (\bfx')| \leq \lambda_r | \bfx - \bfx'| \quad \forall \bfx,\bfx' \in \partial\obs \cap B(\bfx^\ast, r).
\end{cases}
\end{equation}

In the next result we consider an initial position of the game $(E,P)\in \free^2\setminus \target$ such that the line of sight between the players passes close to $\bfx^\ast$.
This assumption can be mathematically expressed by means of the visibility horizons (see Assumption \ref{assump: initial pos} in Section \ref{sec: boundary estimates smooth} for further details).
We denote by $\hl, \hu\in \partial\obs$ the lower and upper visibility horizons from $E$ and $P$ respectively (the symmetric case is analogous), and let us assume that $\hl, \hu\in B(\bfx^\ast, r)$.
We set the quantity
$$
d^\ast (E,P) := \left(\dfrac{\gamma_p}{d_P} - \dfrac{\gamma_e}{d_E}\right)_+
\qquad
\text{where $d_E = |E-\hl|$ and $d_P = |P-\hu|$.}
$$
Here, $(\cdot)_+$ denotes the positive part function, i.e. $(x)_+ = \max \{ 0, x\}$ for any $x\in \R$.

\begin{figure}[t]
    \centering
    \includegraphics[scale = .4]{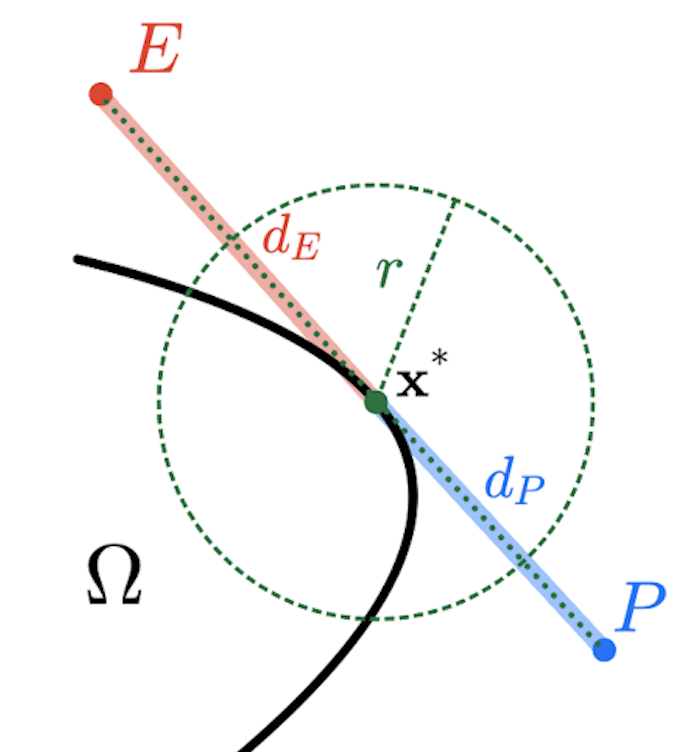}
\hspace{2cm}
    \includegraphics[scale = .5]{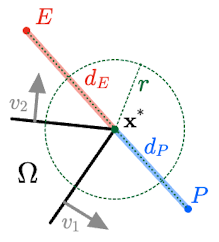}
    \caption{Illustration of two boundary points $(E,P)\in \boundaryT$. At the left,  the horizon point $\bfx^\ast\in \partial\obs$ is in a smooth part of the obstacle, and at the right, it is a corner.}
    \label{fig:boundary points two cases}
\end{figure}

The next result provides an upper and a lower estimate for the value when the initial position $(E,P)$ is close to the usable part of the boundary. 
Note that this is the case when $| \hl - \hu|$ and $d^\ast(E,P)$ are both small.
The quantity $| \hl - \hu|$ being small implies that $(E,P)$ is close to $\boundaryT$, whereas $d^\ast(E,P)$ being small implies that $(E,P)$ is close to the usable part of $\boundaryT$.

\begin{theorem}
\label{thm: smooth obst intro}
Consider the two-player surveillance-evasion game in a two-dimensional environment with one obstacle satisfying \eqref{assum boundary smooth intro} and isotropic and homogeneous dynamics \eqref{homogeneous isotropic intro}.
Assume that the initial position $(E,P)\in \free^2\setminus \target$ is such that the visibility horizons $\hl,\hu\in \partial\obs$ from $E$ and $P$ lie in $B(0,r/2)$.
Then there exist $0<c_1<C_2$ and a positive constant $\varepsilon>0$ depending on $\gamma_e,\gamma_p,d_E,d_P,  r$ and $\lambda_r$ such that, if $|\hl-\hu|<\varepsilon$ and $d^\ast (E,P)<\varepsilon$,  then the value of the game $V(E,P)$ satisfies
$$
c_1\kappa_0 d^\ast(E,P) \leq V(E,P) \leq C_2 (\kappa_0+r\lambda_r)^2 \sqrt{|\hl - \hu|} +  C_2 (\kappa_0+r\lambda_r) d^\ast (E,P).
$$
\end{theorem}

The proof is given in section \ref{sec: boundary estimates smooth} and uses the equations derived in \cite{tsai2004visibility} for the dynamics of the visibility horizons of the players. See Theorem \ref{thm: estimates k const} for a more detailed version of this result, with the explicit expressions for $c_1, C_2$ and $\varepsilon$.
The proof and the statement of Theorem \ref{thm: estimates k const} uses a parametrisation of the boundary $\partial\obs$ (see Assumption \ref{assump: initial pos}) which we have omitted here for presentation purposes.

\begin{remark}[Boundary profile]
\label{rmk: boundary profile}
By letting the initial position converge to the boundary of the target set $(E,P)\to \boundaryT$, which implies $|\hl - \hu| \to 0$, one can use Theorem \ref{thm: smooth obst intro} to estimate the boundary profile of the value (see Corollary \ref{cor: boundary estimates smooth}), i.e. there exists $0<c_1<C_2$ and $d_0>0$ such that
$$
\kappa_0 c_1 d^\ast (E,P) \leq V(E,P)
\leq (\kappa_0+r\lambda_r) C_2 d^\ast (E,P) , \qquad \text{whenever} \  d^\ast (E,P) < d_0.
$$
Moreover, when the curvature $\kappa_0$ is large enough,  we prove in Corollary \ref{cor: boundary estimates smooth sharp} the following sharper estimate:
$$
\kappa_0 C^\ast d^\ast (E,P) - C_0 d^\ast (E,P)^2 \leq V(E,P)
\leq (\kappa_0+r\lambda_r) C^\ast d^\ast (E,P) + C_0 d^\ast (E,P)^2,
$$
where $C^\ast = \frac{2d_P^3 d_E^3}{d_E^3 \gamma_p^2 + d_P^3 \gamma_e^2}$.
This provides explicit estimates of the first-order term of the boundary profile of $V(E,P)$ on $\boundaryT$. Roughly speaking, when the obstacle is smooth at the horizon point $\bfx^\ast$,  and for $(E,P)\in \boundaryT$ close to the usable part of the boundary, the value increases linearly with $d^\ast(E,P)$, with a slope proportional to the curvature.
The estimate is sharp when when the curvature $\kappa(\cdot)$ function is constant  in a neighbourhood of $\bfx^\ast$, i.e., when $\boundaryT$ is locally circular at $\bfx^\ast$.
\end{remark}

The behaviour described above differs drastically from the case when the horizon point $\bfx^\ast\in \partial\obs$ is a corner of the obstacle.
Let us consider now an obstacle $\obs\subset\R^2$ satisfying the following assumption (see Figure \ref{fig:boundary points two cases} (right) for an illustration):
\begin{equation}
    \label{assum horizon corner intro}
    \begin{cases}
   \exists r>0 \ \text{and} \
-\frac{\pi}{2} \leq \theta_1 < \theta_2 <\frac{\pi}{2}
    \  \text{such that} \ v_i = (\cos \theta_i, \sin \theta_i) \ \text{with} \ i=1,2 \   \text{satisfy}  \\
 \obs \cap B(\bfx^\ast, r) = \bigcap_{i=1,2} \{ x\in B(\bfx^\ast, r) \ : \ (x-\bfx^\ast) \cdot v_i < 0  \}.
    \end{cases}
\end{equation}
The main difference resides on the fact that, if one considers a vantage point $P\in \free$  (resp.  $E\in \free$) with visibility horizon at the corner $\bfx^\ast$, then in a neighbourhood of $P$, the visibility horizon is always $\bfx^\ast$.
This makes the use of polar coordinates quite convenient to study this case, considering $\bfx^\ast$ to be the origin without loss of generality.

Given the initial position of the game $(E,P)\in \free^2\setminus \target$, let us denote by $(d_E, \theta_E)$ and $(d_P,\theta_P)$ the corresponding polar coordinates (recall that $d_E = | E - \bfx^\ast|$ and $d_P = |P-\bfx^\ast|$).
Note that $(E,P)$ whenever $|\theta_E - \theta| = \pi$.
We can now state our main result concerning the behaviour of the value of the game near the usable part of $\boundaryT$ when the line of sight between the player is near a corner of $\partial\obs$.

\begin{theorem}
\label{thm: corner obst intro}
Consider the two-player surveillance-evasion game in a two-dimensional environment with isotropic and homogeneous dynamics \eqref{homogeneous isotropic intro}, and an obstacle $\obs$ such that $\partial\obs$ has a corner at $\bfx^\ast = (0,0)$ satisfying \eqref{assum horizon corner intro}.
Assume that the initial position $(E,P)\in \free^2\setminus \target$ expressed in polar coordinates as $(d_E, \theta_E)$ and $(d_P,\theta_P)$ satisfies $\min\{d_E, d_P\} > \underline{d}$, for some $\underline{d}>0$ and 
\begin{equation}
\label{angular coord hyp intro}
\theta_1-\dfrac{\pi}{2} < \theta_E < \theta_2 - \dfrac{\pi}{2},
\qquad
\theta_2-\dfrac{\pi}{2} < \theta_P < \theta_2 + \dfrac{\pi}{2}
\quad \text{and} \quad
\theta_E- \theta_P < \pi.
\end{equation}
Then, there exist $t_0>0$ depending on $\theta_1,\theta_2, \gamma_e,\gamma_p$ and $\underline{d}$ such that
\begin{enumerate}
\item If $\gamma_e d_P \leq  \gamma_p d_E$, then $V(E,P) \geq t_0$.
\item If $\gamma_e d_P > \gamma_p d_E$, then 
$$
V(E,P) \leq    \dfrac{\pi - (\theta_E-\theta_P)}{\frac{\gamma_e}{d_E} - \frac{\gamma_p}{d_P}} \qquad
\text{whenever} \quad \theta_E-\theta_P \geq \pi - t_0 \left( \dfrac{\gamma_e}{d_E} - \dfrac{\gamma_p}{d_P} \right).
$$
\end{enumerate}
\end{theorem}

Note that, when the angular coordinates of the initial position satisfy $\theta_E-\theta_P \sim \pi$, the position $(E,P)$ is close to the boundary of $\target$, hence, the second point in the above theorem gives an upper estimate for the value close to the usable part of $\boundaryT$.
Indeed, it proves that $V(E,P)$ converges to zero as $(E,P)$ tends to the usable part.
The proof is given in Section \ref{sec: boundary estimates corner}, and follows a similar argument to the one in the proof of Theorem \ref{thm: smooth obst intro} but in polar coordinates.

\begin{remark}[Boundary profile]
\label{rmk: jump discontinuity corner}
As in the smooth case, we can let $(E,P)$ converge to the boundary of the target set $\target$ to obtain the profile of the value on $\boundaryT$.
We observe that near the non-usable part of $\boundaryT$, i.e. when $\frac{\gamma_e}{d_E} < \frac{\gamma_p}{d_P}$, the value $V(E,P)$ is uniformly positive,  whilst in the usable part of the $\boundaryT$ it obviously vanishes.
This indicates that on the limit of the usable part of $\boundaryT$, i.e. when $\frac{\gamma_e}{d_E} = \frac{\gamma_p}{d_P}$, the profile of the value exhibits a jump discontinuity.
\end{remark}

The above result shows that, when the initial position of the game $(E,P)\in \free^2\setminus \target$ is close to $\target$, in such a way that the line of sight $[E,P]$ is close to the corner $\bfx^\ast$ (i.e. when $\theta_E-\theta_P\sim \pi$), the value function exhibits a complete different behaviour at either side of the surface $\{\gamma_e d_P = \gamma_p d_E\}$.
Looking at the proof of Theorem \ref{thm: corner obst intro}, this different behaviour can be interpreted as follows: when the initial position satisfies $\gamma_e d_P >\gamma_p d_E$, the evader can use the corner $\bfx^\ast$ to get occluded from the pursuer in a time proportional to $\theta_E-\theta_P$, whereas if $\gamma_e d_P <\gamma_p d_E$, the evader cannot use $\bfx^\ast$ to get occluded.
It does not mean however that the value is infinite in this region, as the evader may be able to use another part of the obstacle to get occluded.

A question that naturally arises is the possibility of the surface $\{\gamma_e d_P = \gamma_p d_E\}$ being a semi-permeable barrier for the game \cite{isaacs1965differential, cardaliaguet1997nonsmooth}.
Semi-permeable barriers are oriented surfaces which enjoy the \textit{semi-permeability property}.
Namely, each player can avoid the state of the system $(E(\cdot) , P(\cdot))$ to cross a semi-permeable barrier in one sense.
If the initial position $(E,P)$ belongs to a semi-permeable barrier $\mathcal{S}:= \{ (E,P) \ : \ g(E,P) = 0\}$, then the evader has a non-anticipative strategy that prevents the game from entering in the set  $\{ (E,P) \ : \ g(E,P) > 0\}$;
and likewise, the pursuer has a non-anticipative strategy that prevents the game from entering in the set $\{ (E,P) \ : \ g(E,P) < 0\}$.
For a  smooth surface, it is well-known that the semi-permeability property is equivalent to 
$$
H(E,P, \nabla_E g(E,P) , \nabla_P g(E,P) ) = 0, \qquad
\forall (E,P) \in \mathcal{S}.
$$

The next result shows that under the assumptions of Theorem \ref{thm: corner obst intro}, there is a semi-permeable barrier emanating from the limit of the usable part of the boundary.

\begin{theorem}
\label{thm: semipermeable barrier}
Consider the two-player surveillance-evasion game in a two-dimensional environment with isotropic and homogeneous dynamics \eqref{homogeneous isotropic intro}, and an obstacle $\obs$ such that $\partial\obs$ has a corner at $\bfx^\ast = (0,0)$ satisfying \eqref{assum horizon corner intro}.
Then the hyper-surface
$$
\mathcal{S} := \{ (E,P)\in \free^2\setminus \target \, : \quad  
\gamma_e d_P = \gamma_p d_E \ \text{and} \ 
(\theta_E,\theta_P) \ \text{satisfying \eqref{angular coord hyp intro}}
\}.
$$
is a semi-permeable barrier.
\end{theorem}

The proof of this theorem is done at the end of Section \ref{sec: boundary estimates corner}. It consists on the construction of the explicit non-anticipating strategies for $E$ and $P$ that prevent the game from trespassing the barrier $\mathcal{S}$.
The possibility of a semi-permeable barrier emanating from the boundary between the usable and the non-usable part of $\boundaryT$ was discussed in \cite{bhattacharya2016visibility}, where they prove that it is not the case when the obstacle $\obs$ is circular.
This result goes in the same direction of the continuous boundary profile that we prove in Theorem \ref{thm: smooth obst intro} for smooth obstacles.
When $\obs$ contains a corner, we show in Theorem \ref{thm: corner obst intro} that the boundary profile exhibits a jump discontinuity at the boundary between the usable and the non-usable part of $\boundaryT$, and moreover, we show in Theorem \ref{thm: semipermeable barrier} that, contrary to the smooth case studied in \cite{bhattacharya2016visibility}, there is a semi-permeable barrier emanating from the limit of the usable part.

%% file: sections/preliminaries.tex
%%% Section: preliminaries

\section{Preliminaries}
\label{sec: preliminaries}

In this section,  we formulate in detail the surveillance-evasion game from the perspective of differential games.
We consider the game with multiple pursuers and evaders in a two-dimensional environment and general dynamics.
The goal is to make precise the  definition of the value of the game and to formulate the boundary-value problem associated to the HJI equation.
% In order to describe the solution of the HJI equation in the following section by means of the level-set method \cite{sethian1996fast,sethian1999fast}  \Carlos{Q: level-set method or fast marching method??}, 
We shall discuss two game paradigms: the finite-horizon game, in which there is an upper limit for the duration of the game, and the infinite-horizon game, in which no time limit for the game is imposed. 
As we will see in section \ref{sec: discontinuities}, the description of the value by means of a propagating front, which in this case is not strictly monotone,  helps understanding the discontinuities and the non-usable part of the boundary.

\subsection{Game set-up}
\label{subsec: game setup}

The game is considered in a domain of the form $\free := \R^2\setminus \obs,$ where $\obs$ is an open, bounded set with Lipschitz boundary in $\R^2$ that we will refer to as the obstacle.
The obstacle $\obs$ may have multiple connected components, which accounts for the case of multiple obstacles.
Although there can be multiple evaders and pursuers, the game can be formulated as a two-player zero-sum game. The first player is represented by the $m$ evaders, denoted by $E = \left(E^{(1)}, \ldots , E^{(m)}\right)$, which try to minimise the occlusion time;
and the second player is represented by the $n$ pursuers, denoted by $P=\left( P^{(1)}, \ldots , P^{(n)} \right)$, which try to maximise it.

Given two compact sets $\mathcal{A}$ and $\mathcal{B}$ of some finite dimensional space, representing the control sets for the evaders and the pursuers respectively, and two functions (see the hypotheses below)
$$
f: \free\times \mathcal{A} \longrightarrow \R^2, \quad
\text{and} \quad
g: \free\times \mathcal{B} \longrightarrow \R^2,
$$
the evolution of the game is given by the system of controlled ODEs
\begin{equation}
\label{game dynamics surv-evasion}
\begin{cases}
\dot{E}^{(i)}(t) = f(E^{(i)}(t), a^{(i)}(t)) & t>0, \quad i\in \{1,\ldots , m\} \\
\dot{P}^{(j)}(t) = g(P^{(j)}(t), b^{(j)}(t)) & t>0, \quad j\in \{1,\ldots , n\} \\
E(0) = E_0 := (E_0^{(1)}, \ldots , E_0^{(m)}) \\
P(0) = P_0 := (P_0^{(1)}, \ldots ,P_0^{(n)}),
\end{cases}
\end{equation}
where $(E_0,P_0)\in \free^m\times \free^n$ represents the initial position of the game,  and
the controls 
$$
a(\cdot) = (a^{(1)}, \ldots , a^{(m)})(\cdot): (0,\infty) \longrightarrow \mathcal{A}^m, \qquad
b(\cdot) = (b^{(1)}, \ldots ,b^{(n)})(\cdot): (0,\infty) \longrightarrow \mathcal{B}^n
$$
are measurable functions chosen by the players.

Since the players are not allowed to enter into the obstacle $\obs$, we define the set of admissible controls for $E$ as
$$ 
\mathcal{C}_e(E_0) := \left\{ a: (0,\infty) \to\mathcal{A}^m \quad \text{measurable s.t.} \  E(t) \in \free^m \ \forall t\in (0,+\infty)  \right\},
$$
and analogously, the set of admissible controls for $P$ is defined as
$$ 
\mathcal{C}_p(P_0) := \left\{ b: (0, \infty) \to \mathcal{B}^n \quad \text{measurable s.t.} \  P(t)\in \free^n \ \forall t\in (0,+\infty)  \right\}.
$$
In other words, the admissible controls are those which make each player stay in the free domain $\free$ forever.
The fact that the obstacle $\obs$ is an open set allows the players to move along the boundary of the obstacle.
This is a typical assumption in shortest path to a target problems with obstacles, which ensures the existence of optimal paths.

In order to ensure that the sets of admissible controls $\mathcal{C}_e (E_0)$ and $\mathcal{C}_p (P_0)$ are non-empty for any $E_0\in \free^m$ and $P_0\in \free^n$, and that the system of ODEs \eqref{game dynamics surv-evasion} is well-posed,  we make the following assumptions:
\begin{equation}
\label{assumption general dynamics}
\begin{cases}
f(E,a) \ \text{and} \ g(P,b) \ \text{are Lipschitz continuous w.r.t.
$E$ and $P$ respectively.} \\
\bigcup_{a\in \mathcal{A}} f(E,a) \ \text{and} \
\bigcup_{b\in \mathcal{B}} f(P,b) \ 
\text{are convex for any $E$ and $P$.} \\
\{ f(E, a)\ : \ a\in \mathcal{A}\} \cap \operatorname{int} (C_\obs (E)) \neq \emptyset \quad \forall E\in \partial\obs. \\
\{ g(P, b)\ : \ b\in \mathcal{B}\} \cap \operatorname{int} (C_\obs (P)) \neq \emptyset \quad \forall P\in \partial\obs. 
\end{cases}
\end{equation}
Here, $\operatorname{int} (C_\obs (E))$ and $\operatorname{int} (C_\obs (P))$ denote the interior of the Clarke's tangent cone to $\free = \R^2\setminus\obs$ at $E$ and $P$ respectively (see \cite[Section 4]{FRANKOWSKA2000449} for the definition).
We use the same framework as in the paper \cite{cardaliaguet2000pursuit} about pursuit differential games with constraints, with the extension to Lipschitz obstacles from \cite{FRANKOWSKA2000449}.
We stress that the surveillance-evasion problem that we consider here can be seen as a particular case of the more general pursuit-evasion games considered in \cite{cardaliaguet1999set,cardaliaguet2000pursuit}.

We recall that the game ends at the first time such that all the evaders are occluded from the pursuers' line-of-sight.
Then, we define the target set $\target\subset \free^{m+n}$ as
\begin{equation}
\label{target}
\target := \{ (E,P)\in \free^m\times \free^n \ : \quad [E_i,P_j]\cap \obs \neq \emptyset \quad \forall (i,j)\in \{1,\ldots ,m \}\times \{1, \ldots , n  \} \}.
\end{equation}
Since we are using the same framework as in \cite{cardaliaguet1999set, cardaliaguet2000pursuit}, 
we need the target to be closed. 
Consequently, given the initial position of the game $(E_0,P_0) \in \free^{m+n}\setminus \target$ and two admissible controls $(a(\cdot),b(\cdot))\in \mathcal{C}_e(E_0)\times \mathcal{C}_p (P_0)$ we define the end-game time as the first time such that the position of the game reaches the boundary of $\target$, i.e.
\begin{equation}
\label{end-game time}
t^\ast (E_0, P_0, a(\cdot), b(\cdot)):=
\min\{   
t\geq 0 \quad \text{such that} \quad (E(t), P(t))\in \boundaryT
\}.
\end{equation}
As mentioned in the introduction, $\boundaryT$ refers to the boundary of $\target$ in the topology relative to $\free^{m+n}$, i.e. the topology generated by the open balls in $(\R^2)^{m+n}$ restricted to $\free^{m+n}$.
In this way,  the game does not end when a player reaches the boundary of the obstacle, but only when they reach the non-visibility region (see Figure \ref{fig:partialD} for an illustration).
We assume the convention $t^\ast (E_0,P_0, a(\cdot), b(\cdot))  = +\infty$ whenever $E_0, P_0, a(\cdot)$ and $b(\cdot)$ are such that $(E(t), P(t))\not\in \boundaryT$ for all $t\geq 0$.
And naturally, if $(E_0,P_0)\in \boundaryT$, then $t^\ast(E_0,P_0,a(\cdot), b(\cdot))=0$ for any $a(\cdot)\in C_e(E_0)$ and $b(\cdot)\in C_p (P_0)$.

We now use the definition of the first hitting time $t^\ast(\cdot)$ to introduce the payoff of the game. 
We consider two game paradigms: the finite-horizon problem and the infinite horizon one.
Given a time-horizon $t>0$, the pay-off of the finite-horizon game is defined as
\begin{equation}
    \label{payoff finite horizon}
    J_t( E_0, P_0, a(\cdot), b(\cdot))
    :=
    \min \{ t, \, t^\ast (X_0, a(\cdot) , b(\cdot)) \}.
\end{equation}
For the infinite-horizon problem, the payoff is given by the first hitting time, i.e.
\begin{equation}
    \label{payoff infinite horizon}
    J (X_0, a(\cdot), b(\cdot))
    := t^\ast (E_0, P_0, a(\cdot), b(\cdot)).
\end{equation}

As we mentioned in the introduction, we consider the surveillance-evasion game in which the evaders seek to get occluded from the pursuers' line of sight in the shortest time possible,  whereas the pursuers' goal is to prevent occlusion.
Assuming that the group of evaders and the group of pursuers play cooperatively (each group of players playing as a single player), one can consider this as a zero-sum two-player game, in which one player controls the group of the evaders $E$, with the goal of minimizing the pay-off function  $J_t(\cdot)$ (resp. $J(\cdot)$), and the other player controls the group of pursuers $P$ with the goal of maximizing the same pay-off.

\subsection{The value of the game and the Hamilton-Jacobi-Isaacs equation}
\label{subsec: HJI eq}

The value of the game represents the best payoff that a player can ensure, assuming an optimal response by the opponent.
The value can therefore be defined in two ways: from the perspective of either the Evader or the Pursuer. These are typically referred to as the lower and the upper value of the game respectively.
When both values coincide, we say that the game has value.
For simplicity, we shall always consider the lower value, i.e. from the perspective of the Evader.
See \cite{cardaliaguet2000pursuit} for a proof of existence of value for this game, and a further discussion about the relation between the upper and the lower value of pursuit-evasion games with state constraints.

In order to define the lower value of the game, we need to give an appropriate definition of strategy for the player $E$. 
We use the notion of non-anticipating strategies following the definition of Elliot and Kalton \cite{ elliott1972existence2, elliott1972existence}. In particular, we use the exact same definition of strategy as in \cite{cardaliaguet1999set, cardaliaguet2000pursuit}, which deal with pursuit-evasion games in the case of state-constrained dynamics.
Given the initial condition $(E_0, P_0)\in \free^{m+n}\setminus \target$, the set of non-anticipating strategies for $E$ is defined as
$$
\begin{array}{ll}
\mathcal{S}_e (E_0, P_0) := \{\alpha: \mathcal{C}_p (P_0) \to \mathcal{C}_e (E_0) \ : & \text{if}\ \exists \tau >0 \ \text{s.t.} \ b_1(t) = b_2(t) \ \text{for a.e.}\ t\in (0,\tau), \\ 
& \text{then}\ 
\alpha [b_1] (t) = \alpha[b_2] (t) \ \text{for a.e.}\ t\in (0,\tau)\} .
\end{array}
$$
With this notion of strategy, for any $t>0$ and $(E,P)\in \game$, we define the lower value of the finite-horizon game  as
\begin{equation}
\label{value function finite h}
\tilde{V} (t, E,P) := \inf_{\alpha \in \mathcal{S}_e(E,P)}
\sup_{b\in \mathcal{C}_p(P)}
J_t (E,P, \alpha[b] (\cdot), b(\cdot)).
\end{equation}
Note that $\tilde{V} (t,E,P)$ is always finite. Indeed, it holds that
$$
0 < \tilde{V}(t,E,P) \leq t, \qquad \forall (t, E,P)\in (0,\infty)\times (\free^{m+n}\setminus \target).
$$
On the contrary, the value of the infinite-horizon game, defined as
\begin{equation}
\label{value function infinite h}
V (E,P) :=  \inf_{\alpha \in \mathcal{S}_e(E,P)}
\sup_{b\in \mathcal{C}_p(P)}
J_\infty (E,P, \alpha[b] (\cdot), b(\cdot)),
\end{equation}
can take infinite values. 

In both cases, we observe that, by the definition of the first-hitting time $t^\ast(\cdot)$ in \eqref{end-game time}, the value of the game vanishes on the boundary of the target set $\target$. That being said,  as we announced in the introduction,   the value of the game develops discontinuities near some parts of $\boundaryT$, the so-called non-usable part of the boundary. 
Namely, there exist parts of $\boundaryT$ which are not reachable by the game, assuming perfect play by the pursuers. Hence, even if the initial position of the game is arbitrarily close to those parts of the boundary, the value of the game remains uniformly positive.

\begin{remark}
\label{rmk:relation finit-inf horizon}
    One can readily prove that, for any $t>0$, the values of the finite-horizon and the infinite-horizon games satisfy the relation
    $$
    \tilde{V}(t, E,P) := \min \{ t, \, 
    V(E,P)\}.
    $$
    At the points $(E,P) \in \free^{m+n}\setminus \target$ where $V (E,P)$ is finite, the function $t\to \tilde{V}(t, E,P)$ grows at speed one until $t= V (E,P)$, and then stays constant. At the points $(E,P)\in \game$ where $V (E,P)$ is infinite, we have $\tilde{V}(t,E,P) = t$ for all $t>0$.
\end{remark}

Next we introduce the Hamilton-Jacobi-Isaacs equation associated to the differential game that we are considering.
This is the partial differential equation that characterises the value of the game.
For the finite-horizon game, one obtains a time-evolution first-order PDE coupled with initial and boundary conditions,  posed on the domain  $(0,\infty)\times \game$, where $\game$ is defined as
$$
\game := \free^{m+n}\setminus \target,
$$
and will be referred to as the game domain.
For the infinite-horizon problem,  one obtains a time-independent first-order PDE with boundary condition, posed in the game domain $\game$.

%\begin{remark}[About the boundary of $\game$]
    As before, the boundary of $\game$, denoted by $\boundary$, refers to the boundary of $\game\subset \free^{m+n}$ in the topology relative to $\free^{m+n}$, i.e. the topology generated by the open balls in $(\R^2)^{m+n}$ intersected with $\free^{m+n}$. 
    Hence, points on the boundary of the obstacle $\obs$ do not in general belong to $\boundary$. See Figure \ref{fig:partialD} for an illustration.
   Note that it holds that $\boundary= \boundaryT$.
%\end{remark}

In order to ease the notation, we denote from now on the position of the game by $X:= (E,P)\in\game$.
The dynamics \eqref{game dynamics surv-evasion} associated to the evolution of the game will be written in the more compact form
$$
\begin{cases}
\dot{X} (t) = F(X(t), a(t), b(t)) & t>0, \\
X(0) = X_0,
\end{cases}
$$
where $X_0\in \game$ denotes the initial position of the game.

Using the dynamic programming principle, Isaacs proved in \cite{isaacs1965differential} that, at points $(t,X) \in (0,\infty)\times \game$ where $\tilde{V}(\cdot)$ is differentiable, it satisfies the first-order partial differential equation
\begin{equation}
\label{HJI evol}
\partial_t \tilde{V} (t,X) + H(X,\nabla_X \tilde{V}(t,X) ) = 1,
\end{equation}
where the Hamiltonian $H:\game\times (\R^2)^{m+n} \to \R$ is given by
\begin{equation}
\label{hamiltonian preliminaries}
H(X, \rho) := \sup_{a\in \mathcal{A}} \inf_{b\in \mathcal{B}} -F(X, a,b) \cdot \rho.
\end{equation}
Similarly, at points $X\in\game$ where $V(\cdot)$ is differentiable, it holds that
\begin{equation}
\label{HJI infinite}
 H(X,\nabla_X V (X) ) = 1.
\end{equation}
The partial differential equations \eqref{HJI evol} and \eqref{HJI infinite} are known as the Hamilton-Jacobi-Isaacs equations associated to the finite-horizon game and the infinite-horizon game respectively.

%\begin{remark}
Using the expanded notation for the position of the game $X=(E,P)$ as
$E = (E^{(1)}, \ldots , E^{(m)})$ and $P=(P^{(1)}, \ldots,P^{(n)})$,
and for the dynamics in \eqref{game dynamics surv-evasion},  the Hamiltonian reads as follows:
\begin{equation}
\label{hamiltonian explicit}
H(E,P, \rho_E,  \rho_P) = \sum_{i=1}^m \sup_{a\in \mathcal{A}} \left\{ - f(E^{(i)} ,  a)\cdot \rho_E^{(i)} \right\}
+ \sum_{j=1}^n \inf_{b\in \mathcal{B}} \left\{ -g(P^{(j)},  b) \cdot \rho_P^{(j)} \right\},
\end{equation}
where $\rho_E = \left( \rho_E^{(1)}, \ldots, \rho_E^{(m)} \right) \in \R^{md}$ and 
$\rho_P = \left( \rho_P^{(1)}, \ldots, \rho_P^{(n)} \right) \in \R^{nd}$.
Note that the gradient of the value $V (X)$ with respect to the game position can be written as $\nabla_X V (X) = (\nabla_E V (E,P),  \nabla_P V(E,P))$, where
$$
\nabla_E V (E,P) = \big( \nabla_{E^{(1)}} V(E,P), \ldots ,  \nabla_{E^{(m)}} V(E,P)  \big) \in \R^{md}
$$
and
$$
\nabla_P V (E,P) = \big(  \nabla_{P^{(1)}} V(E,P), \ldots ,  \nabla_{P^{(n)}} V(E,P)  \big) \in \R^{nd}.
$$
The gradient $\nabla_X V(t,X)$ is is defined analogously.
From the explicit expression of the Hamiltonian in \eqref{hamiltonian explicit}, it is easy to check that $H(X,\rho)$ in \eqref{hamiltonian preliminaries} satisfies the Isaacs condition, i.e., the sup and the inf are interchangeable.
%\end{remark}

The partial differential equations \eqref{HJI evol} and \eqref{HJI infinite} give a characterisation of the value of the game when it is differentiable. However, the value for this kind of games typically fails to be $C^1$, and in our case, the situation is even worse as it develops discontinuities in most of the cases. The equations \eqref{HJI evol} and \eqref{HJI infinite}  have to be understood in the sense of viscosity solutions \cite{crandall1984some,crandall1992user,crandall1983viscosity}.
Although initially they were developed to deal with continuous solutions to Hamilton-Jacobi equations, the theory of viscosity solutions has been adapted to deal with discontinuous solutions as well (see for instance \cite{Bardi1997,ishii1987perron}).
It is proved in \cite[Theorem 7.5]{cardaliaguet1999set} that the function $\tilde{V}(t,X)$, as defined in \eqref{value function finite h}, is the smallest lower semicontinuous viscosity supersolution to the initial-boundary value problem
\begin{equation}
\label{IBPHJI}
    \begin{cases}
        \partial_t \tilde{V} + H(X, \nabla \tilde{V})  = 1  & \text{in} \ (0,\infty)\times \game,  \\
\tilde{V} (t, X) = 0 &  \text{on}\ (0,\infty)\times  \boundary,\\
\tilde{V}(0, X) = 0 &  \text{in} \  \boundary.
    \end{cases}
\end{equation}
Similarly, the function $V (X)$, as defined in \eqref{value function finite h}, is the smallest lower semicontinuous viscosity supersolution to the boundary value problem
\begin{equation}
\label{BP HJI}
    \begin{cases}
      H(X, \nabla V)  = 1  & \text{in} \ (0,\infty)\times \game, \\
V(X) = 0 &  \text{in} \  \boundary,
    \end{cases}
\end{equation}
where in both cases, the Hamiltonian $H(X,\rho)$ is given by \eqref{hamiltonian preliminaries}.

Despite being interesting from a theoretical viewpoint as a characterisation of the value of the game, the problems \eqref{IBPHJI} and \eqref{BP HJI} are of little use in practice if one is not able to compute (or numerically approximate) the correct viscosity solution.
Moreover, one needs to do so without a complexity that scales exponentially with $(m+n)d$.

For a class of first-order Hamilton-Jacobi equations (e.g. the Eikonal equation arising in shortest path problems), the so-called Fast Marching method \cite{sethian1996fast,sethian1999fast,tsitsiklis1995efficient,helmsen1996two,sethian2001ordered} can be used to efficiently compute solutions. These methods rely on the
strict monotonicity of the solution along the characteristics, and use a sorting strategy to update solutions on grid nodes in essentially a single pass.
Although our problem is very much similar to a shortest path problem, the fact that one player is adversarial prevents the propagation of the front from being strictly monotone, which is a big inconvenience when implementing Fast Marching methods.  We will see in the following section that this lack of monotonicity of the front propagation is the reason for the discontinuities of the solution.

The Fast Sweeping methods, \cite{zhao2005fast,tsai2003fast,kao2005fast}, on the other hand, rely on Gauss-Seidel type iterations. These algorithms utilise different orderings of the grid nodes and upwind discretisation to propagate families of the characteristics efficiently.  It can be proved that first order approximation to the viscosity solutions can be computed with a constant number of iterations.

There are several numerical methods that can be used to approximate the value of the game by solving a discrete approximation of the above boundary-value problem.
The semi-Lagrangian schemes \cite{bardi1995convergence, Bardi1999, falcone2006numerical} are iterative methods based on finite differences that make use of the dynamic programming principle to compute the value of a discrete version of the game defined on a grid. 
These schemes are proven to converge to the viscosity solution in sets where this one is continuous, even though it can be also implemented in the presence of discontinuities. 
Recently, a fast sweeping scheme was developed in \cite{ly2020visibility} for solving the HJIs that we discuss in this paper. However, all the aforementioned methods suffer from the curse of dimensionality because they rely on using grids to discretise the equation and propagate the causality in some way.

Finally, we point out that there has been rapid development of numerical methods for solving Hamilton-Jacobi equations in higher dimensions \cite{darbon2016algorithms,chow2018algorithm,chow2019algorithm} in the past few years. These methods do not rely on discretisation of the Hamiltonian on a mesh, but instead, characterisations of the solutions, such as the Hopf-Lax formula. However, the methods work for relatively simple boundary geometries and conditions, and it is therefore not clear if they could be applied to the problems analyzed in this paper.

%% file: sections/discontinuities.tex
%% Section: discontinuities and level set

\section{Discontinuities and usable part of the boundary}
\label{sec: discontinuities}

It is well-known that, for pursuit-evasion games with deterministic dynamics such as the one we are considering here, the value of the game develops discontinuities.
One way to prove the existence of such discontinuities is by describing the value of the game as the evolution of a front $\front (t)\subset \game$ that, initially is set to be the boundary of the target set $\boundaryT$, and then evolves in $\game$ in such that way that
$\front (t)$ is the $t$-level set of the value $V(E,P)$.
See \cite{doi:10.1137/0524066,GIBOU201882,sethian1996theory} and the references therein for a review on level-set methods and \cite{evans1984differential, tsai2003level} for applications in differential games with discontinuous solutions. 
The evolution of the front can be described by means of the Hamiltonian $H(X,\rho)$ defined in \eqref{hamiltonian preliminaries}.
This is the key idea of the Fast Marching method \cite{sethian1996fast, sethian1999fast,tsitsiklis1995efficient,helmsen1996two} to solve the Eikonal equation, which is the Hamilton-Jacobi-Bellman associated with the shortest-path problem for a given target region.

From the perspective of the evaders (see the definition of the value in \eqref{value function finite h} and \eqref{value function infinite h}), our surveillance-evasion game can actually be seen as a shortest-path problem to the target $\target$, defined as the non-visibility region.
The particularity is that part of the state of the system is controlled by the pursuers, who try to delay the arrival time as much as possible.
The evaders have to optimise their strategy under the assumption that the pursuers respond optimally to their strategy.
If the pursuers were not mobile, i.e. if $g(P,b) \equiv 0$ in \eqref{game dynamics surv-evasion}, then the problem would be precisely a shortest path problem, and the associated HJI equation would be an Eikonal equation.
In this case, under mild assumptions on the dynamics $f(E,a)$, the front $\front(t)$ can be proven to be strictly monotone in $t$, and the value would be continuous.

As we shall see in subsection \ref{subsec: level-set}, the fact that the pursuers are mobile
prevents the front $\front(t)$ from being strictly monotone, and from this, it can be deduced that the value $V(E,P)$ is discontinuous.
In subsection \ref{subsec: usable part}, we shall pay special attention to the evolution of the front at time $t=0$.
The monotonicity of the front at $t=0$ determines the discontinuities arising on the boundary of the target set $\target$, and thus, the usable and non-usable part of $\boundaryT$.

\subsection{Level-set approach and discontinuities}
\label{subsec: level-set}

For any $t\geq 0$, we define the set
$$
\levelset (t) := \{ X\in \game \, : \quad \tilde{V}(t,X) = t\} = \{ X\in \game \, : \quad V (X) \geq t\}.
$$
The latter identity holds by virtue of Remark \ref{rmk:relation finit-inf horizon}.
We also define the front $\front(t)$ as the boundary\footnote{As everywhere throughout the paper, the boundary is taken with respect to the topology relative to $\free^{m+n}$.} of $\levelset(t)$, i.e.
\begin{equation}
\label{front def}
\front (t) = \partial\levelset(t).
\end{equation}
Observe that $\Gamma(0) = \boundary$, and since $\game = \free^{m+n}\setminus \target$, we have $\Gamma(0)=\boundaryT$.

The set $\levelset (t)$ is clearly non-increasing in the sense that
\begin{equation}
\label{level set monotonicity}
\levelset (t') \subset \levelset (t) \subset \levelset (0) = \game \qquad \forall 0 < t < t',
\end{equation}
and using the lower semicontinuity of $V(\cdot)$, it holds that
\begin{equation}
\label{V infty level set bound}
V (X) \leq t\qquad \forall X\in \front (t).
\end{equation}
For any $t>0$ fixed, the function $X\mapsto \tilde{V}(t,X)$ is constant in $\levelset(t)$.
Moreover,  the Hamiltonian defined in \eqref{hamiltonian explicit} satisfies
$H (X,0) = 0$ for all $X\in \game$.
Therefore, using \eqref{HJI evol}, for every $X$ in the interior of $\levelset (t)$ we have $\partial_t \tilde{V} (s, X) = 1$ for all $s\in (0, t]$.

The above reasoning implies that $\tilde{V}(s,X) = s$ for all $s\in (0,t]$, where $t$ is the first time $t>0$ such that $X\in \front(t)$, and thereafter we have $V (X) = \tilde{V}(s,X) = t$ for all $s\geq t$.
We can therefore describe the function $V$ in terms of the propagation of the front $\front (t)$ as
$$
V (X) = \inf \{ t\geq 0 \, : \quad X\in \front (t)\}.
$$
From this formula, we can deduce (see Lemma \ref{lem: discont front} below) that if the front $\front(t)$ is not strictly monotone (meaning that it has stationary parts), then the function $V(X)$ develops discontinuities.

\begin{lemma}
\label{lem: discont front}
Consider the surveillance-evasion game presented in subsection \ref{subsec: game setup}, let $V(X)$ be the value of the game as defined in \eqref{value function infinite h}, and let $\front (t)$ be the $t$-level set as defined in \eqref{front def}.
For any $X\in \game$, if there exists $0\leq t_1 < t_2$ such that $X\in \front (t)$ for all $t\in [t_1, t_2]$, then $V(\cdot)$ is discontinuous at $X$.
\end{lemma}

\begin{proof}
    In view of \eqref{V infty level set bound}, one has $V(X)\leq t_1$. On the other hand, since $X\in \front(t_2)$, for any $\varepsilon>0$, we have that $\sup \{ V(Y)\ : \ Y\in B(X,\varepsilon)\} \geq t_2$, which implies that $V(X)$ is discontinuous at $X$.
\end{proof}

The propagation of the front can be described by means of the Hamiltonian $H(X,\rho)$ defined in \eqref{hamiltonian preliminaries} as follows:
every point $X\in \front(t)$, the front is propagated at unit speed in the direction
$$
v = \nabla_\rho H (X, \bfn),
\qquad \text{where $\bfn$ is the normal inner vector to $\levelset(t)$ at $X$.}
$$
Recall that $\front (t) = \partial \levelset(t)$.
Due to the property \eqref{level set monotonicity}, the front is propagated only if $v$ points towards the interior of $\levelset (t)$, i.e. if $\nabla_\rho H (X, \bfn)\cdot \bfn > 0$.
Moreover, by the form of the Hamiltonian in \eqref{hamiltonian preliminaries} and the assumptions in \eqref{assumption general dynamics}, it is easy to verify that
\begin{equation}
\label{hamiltonian grad prop}
\nabla_\rho H (X, \rho)\cdot \rho =  H(X,\rho), \qquad
\forall \rho\in \R^{m+n}.
\end{equation}
This can be utilised to characterise discontinuities of $V(\cdot)$ in terms of the sign of the Hamiltonian.

\begin{proposition}
\label{prop: discont Hamiltonian}
Consider the surveillance-evasion game presented in subsection \ref{subsec: game setup}, let $V(X)$ be the value of the game as defined in \eqref{value function infinite h}, $\front (t)$ the $t$-level set as defined in \eqref{front def}, and $H(X,\rho)$ the Hamiltonian defined in \eqref{hamiltonian preliminaries}.
If there exist $t\geq 0$,  $X\in \front (t)$ and $r>0$ such that
the part of the front $\Gamma(t)\cap B(X,r)$ is a smooth hyper-surface and 
$$
H (Y, \bfn (Y)) \leq 0, \qquad
\forall Y\in \Gamma (t) \cap B(X, r),
$$
where $\bfn (Y)$ is the inner normal vector to $\levelset (t)$ at $Y$, then $V(\cdot)$ is discontinuous on $\Gamma (t) \cap B(X, r)$.
\end{proposition}

\begin{proof}
The proof follows from the monotonicity property of $\levelset(\cdot)$ in \eqref{level set monotonicity}, which implies that the front is propagated at $Y\in \front(t)$ only if
$$
\nabla_\rho H (Y, \bfn(Y))\cdot \bfn(Y) > 0.
$$
Using \eqref{hamiltonian grad prop}, we deduce that the front is propagated only at points $Y\in \Gamma (t)$ satisfying 
$H (Y, \bfn (Y)) >0.$
Hence, by the hypothesis of the proposition, we have $\Gamma (t)\cap \Gamma (t+\varepsilon) \cap B(X, r)\neq \emptyset$, for $\varepsilon>0$ small enough, which by Lemma \ref{lem: discont front} implies that $V(\cdot)$ is discontinuous at $\Gamma (t) \cap B(X, r)$.
\end{proof}

\subsection{The usable part of the boundary}
\label{subsec: usable part}

The conclusion of Proposition \ref{prop: discont Hamiltonian} can be used to characterise the usable part of the boundary of the target set $\target$.
The usable part of $\boundaryT$, denoted by $\usable$, is the set of positions $(E,P)\in \boundary$ such that,
if the initial position of the game is sufficiently close to $(E,P)$, then the evaders have a non-anticipating strategy that ends the game almost immediately, no matter the control implemented by the pursuers. On the other hand, if $(E,P)\in \boundaryT$
is in the non-usable part, then there exists a small neighbourhood of $(E,P)$ such that, if the initial position is in this neighbourhood, then
the pursuers have a non-anticipating strategy which prevents the state from hitting $\boundaryT$ for a uniformly positive time.
Taking into account that, by definition, the value of the game $V(E,P)$ vanishes on the boundary of the target set $\boundaryT$, we can define the usable part $\usable$ as the points where $V(\cdot)$ is continuous, and the non-usable part as the points where it is not.

We recall from the definition of the front $\front(t)$ in the previous section, see \eqref{front def}, that $\Gamma(0) = \boundary = \boundaryT$.
Moreover, in our case, the boundary of the target region $\target$, which is defined in \eqref{target} as the non-visibility set is smooth.
Hence, Proposition \ref{prop: discont Hamiltonian}, applied at $t=0$, provides a characterisation of the usable part of $\boundaryT$.

\begin{corollary}
\label{cor: usable part general}
Consider the surveillance-evasion game presented in subsection \ref{subsec: game setup}, and let $H(X,\rho)$ be the Hamiltonian defined in \eqref{hamiltonian preliminaries}.
Let $X\in \boundaryT$ be such that $\boundaryT$ is smooth at $X$, and let $\bfn (X)$ be the normal inner vector to $\game$, or equivalently, the outer normal vector to $\target$. 
Then we have the following:
\begin{enumerate}
    \item If $H (X, \bfn (X)) > 0$, then $X$ is in the usable part.
    \item If $H (X, \bfn (X)) < 0$, then $X$ is in the non-usable part.
\end{enumerate}
\end{corollary}

Let us end the section by considering the simple case with two players (one evader and one pursuer) and the isotropic and homogeneous dynamics in \eqref{homogeneous isotropic intro}.
Recall that the parameters $\gamma_e$ and $\gamma_p$ represent the maximum speeds for the evaders and the pursuer respectively.
In this case, the Hamiltonian \eqref{hamiltonian explicit} can be written explicitly as
$$
\rho:= (\rho_E, \rho_P)\in \R^2\times \R^2 \longmapsto H(\rho_E, \rho_P) = \gamma_e| \rho_E| - \gamma_p |\rho_P|.
$$
If $(E,P)\in \boundaryT$ is such that the segment $[E,P]$ is tangent to $\obs$ at a single point $\bfx^\ast\in \partial\obs$, then $\boundaryT$ is differentiable at $(E,P)$,  and applying Corollary \ref{cor: usable part general}, we obtain that
\begin{equation}
\label{usable simple case}
\text{if}\quad \dfrac{\gamma_e}{|E-\bfx^\ast|} > \dfrac{\gamma_p}{|P-\bfx^\ast|},
\quad \text{then $(E,P)$ is in the usable part,} 
\end{equation}
and
\begin{equation}
\label{non-usable simple case}
\text{if}\quad \dfrac{\gamma_e}{|E-\bfx^\ast|} < \dfrac{\gamma_p}{|P-\bfx^\ast|},
\quad \text{then $(E,P)$ is in the non-usable part.}
\end{equation}
In the two following sections we analyse the behaviour of the value in the interface between these two regimes, i.e. when $\gamma_e |P-\bfx^\ast| = \gamma_p |E-\bfx^\ast|$. Interestingly enough, this one depends drastically on the regularity of the obstacle $\obs$.

%% file: sections/boundary-smooth.tex
\section{Boundary estimates for smooth obstacles}
\label{sec: boundary estimates smooth}

In this section, we study the behaviour of the value of the game near the usable part of the boundary of the target set $\target$ (denoted by $\usable$).
In this case, we consider boundary points $(E,P)\in \boundaryT$ such that the line-of-sight between $E$ and $P$ is tangent to $\obs$ at a single point $\bfx^\ast$ where $\partial\obs$ is smooth (see the condition \eqref{assum boundary smooth intro}).
The goal is to prove Theorem \ref{thm: smooth obst intro} stated in Section \ref{sec: main results}, which applies to the two-player surveillance-evasion game (with one evader and one pursuer) on a two-dimensional environment with isotropic and homogeneous dynamics \eqref{homogeneous isotropic intro}.
As a consequence of this result we obtain, in Corollary \ref{cor: boundary estimates smooth}, estimates for the profile of the value along the boundary of the target set (see Remark \ref{rmk: boundary profile}).
In particular, we describe how the value transitions from being zero on $\usable$ to being positive on $\boundaryT\setminus \usable$.

At the end of this section, we state and prove Theorem \ref{thm: estimates k const}, which is a more explicit version of Theorem \ref{thm: smooth obst intro}.
The arguments for proving Theorem \ref{thm: estimates k const} rely on the analysis of the dynamics of the horizon-points of the players' visibility. For a vantage point in $P\in \free$ (resp. $E$), the horizon points determine the part of the boundary of the obstacle which is visible from $P$. The dynamics of the horizon points were first studied in \cite{tsai2004visibility}. Here, we use the equations derived in \cite{tsai2004visibility} to obtain estimates for the value of the game when the initial position is close to the boundary of the target set, i.e., the line-of-sight between $E$ and $P$ is close to the obstacle.
Moreover, in view of condition \eqref{assum boundary smooth intro}, we assume that part portion of $\partial\obs$ which is closer to the line of sight is uniformly convex and smooth.
Let us mathematically describe the assumption on the initial position of the game $(E,P)\in \game$.
 Let us recall that the game domain is defined as $\game :=\free^2\setminus\target$.

\begin{assumption}[Initial position of the game]
\label{assump: initial pos}
We assume that there is a connected portion of the boundary $\partial \obs_L\subset \partial\obs$, which is smooth and uniformly convex.
More precisely, there exists an arc-length parametrisation  $\Sigma : [-L,L] \to \partial\obs_L$ which is twice differentiable, and such that the curvature  $\kappa (\mathbf{s}) := |\Sigma''(\mathbf{s})|$ is uniformly positive and Lipschitz continuous, i.e.
\begin{equation}
    \label{hyp smooth and convex}
    \left|\kappa (\mathbf{s})\right| \geq \kappa_0
    \quad \text{and} \quad
    \left|  \kappa (\mathbf{s}_1) - \kappa (\mathbf{s}_2) \right| \leq C_L \left| \mathbf{s}_1 - \mathbf{s}_2  \right|, \qquad \forall \mathbf{s}, \mathbf{s}_1, \mathbf{s}_2\in [-L,L],
\end{equation}
for some $\kappa_0>0$ and $C_L\geq 0$. Here, we include the case $C_L=0$, which accounts for the case of $\partial \obs_L$ being the arc of a circle.

As for the initial position of the game $(E,P)\in \game$, we assume that it satisfies the following:
\begin{enumerate}
\item The upper-horizon of the pursuer's visibility, denoted by $\hu$, and the lower-horizon for the evader's visibility, denoted by $\hl$, are both contained in $\partial\obs_L$. In other words, 
if we define
$$
\hus = \Sigma^{-1} (\hu)
\quad \text{and} \quad
\hls  = \Sigma^{-1} (\hl),
$$
then the part of $\partial\obs_L$ visible from $E$ and $P$ is
$$
\{ \Sigma (s) \, : \quad s\in (\hls, L ] \}
\quad \text{and} \quad
\{ \Sigma (s) \, : \quad s\in [-L, \hus)  \},
\quad  \text{respectively.}
$$
\item The visible part of $\partial\obs_L$ from $E$ and $P$ have nonempty intersection, i.e.
$$
\hls < \hus,
$$
which implies that the straight segment joining $E$ and $P$ is at positive distance from $\partial\obs_L$.
See Figure \ref{fig:initial pos} for an illustration of the initial position of the game.
\end{enumerate}
\end{assumption}

\begin{figure}
    \centering
    \includegraphics[scale=.65]{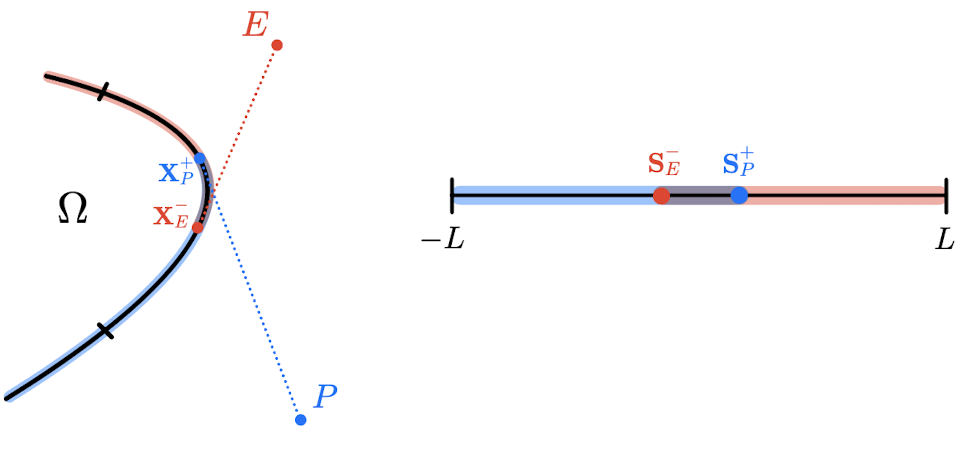}
    \caption{Illustration of the assumption about the initial position of the game described in Assumption \ref{assump: initial pos}. The roles of $E$ and $P$ are interchangeable by changing the sign of the parametrisation.}
    \label{fig:initial pos}
\end{figure}

Given an initial position $(E,P)\in \game$ satisfying Assumption \ref{assump: initial pos} and a game trajectory in a time interval $(0,t_0]$, represented by
$$
t\in (0,t_0] \longmapsto (E(t), P(t)),
$$
it is not difficult to deduce that, if $t_0>0$ is small enough (to make precise later in \eqref{t_0 small cond}), then the game ends in the interval $(0,t_0]$ if and only if there exists $t^\ast\in (0,t_0]$ such that 
$$
\hls (t^\ast) = \hus (t^\ast),
$$
where $\hls (t) = \Sigma^{-1} (\hl(t))$ and $\hus (t) = \Sigma^{-1} (\hu (t))$, and $\hl (t)$ and $\hu (t)$ are the lower and upper visibility horizons for $E(t)$ and $P(t)$ respectively, for all $t\in (0,t_0].$
Here, $t_0$ must be so small that the players don't have enough time to move their visibility horizons out of the parametrised portion of the boundary $\partial\obs_L$.

The main ingredient to obtain estimates for the value of the game is to analyse, for any $t\in (0,t_0]$, the reachable sets of $\hls (t)$ and $\hus (t)$ from the initial position of the game.
In other words, the reachable set of the visibility horizons, expressed using the parameter of the parametrisation $\Sigma(\cdot)$.
Since we are considering that the dynamics are homogeneous and isotropic as in \eqref{homogeneous isotropic intro},  for any $t>0$ the reachable sets in $\free$ from the initial positions $E$ and $P$ are contained in the closed balls $\overline{B(E, \, \gamma_e t)}$ and $\overline{B(P, \, \gamma_p t)}$ respectively.
In particular, if $t>0$ is so small that the players don't have enough time to reach the boundary of the obstacle $\obs$, then the reachable sets from $E$ and $P$ are precisely the aforementioned closed balls.
Given $(E,P)\in \game$ satisfying Assumption \ref{assump: initial pos} and $t>0$ small enough, we can define, for any $t\in [0,t_0]$ and $v\in \overline{B(0,1)}$, 
\begin{equation}
\label{def horizon limits t,v}
\hls (t,v) = \Sigma^{-1} \big( \hl (t,v) \big)
\quad \text{and} \quad
\hus (t,v) = \Sigma^{-1} \big( \hu (t,v) \big),
\end{equation}
where
$\hl (t,v)\in \partial \obs_L$ is the lower-horizon for the visibility from the vantage point $E+\gamma_e t v$, and analogously,
$\hu (t,v)\in \partial \obs_L$ is the upper-horizon for the visibility from the vantage point $P+\gamma_p t v$.
Due to the continuity of the dynamics and the parametrisation $\Sigma(\cdot)$, both functions $\hls (t,v)$ and $\hus (t,v)$ are continuous.

Let us give a sketch of the proof of Theorem \ref{thm: estimates k const}.
\begin{enumerate}
\item In Proposition \ref{prop: value represen formula}, we provide a representation formula for the value of the game, as the smallest time such that the function
\begin{equation}
    \label{S def}
S (t) = \max_{v\in \overline{B(0,1)}} \hus (t,v) - \max_{v\in \overline{B(0,1)}} \hls (t,v)
\end{equation}
vanishes.  Note that $S(t)$ is the difference between the value function of two optimal control problems, consisting on the maximisation of the upper and lower visibility horizons at time $t$ from a given initial position $P$ and $E$ respectively.
\item In Lemma \ref{lem: max horizons}, we characterise the maximisers of the optimal control problems appearing in the function $S(t)$.
\item In the Appendix \ref{sec: lemmas boundary}, we use Lemma \ref{lem: max horizons}, along with the equations for the dynamics of the visibility horizons from \cite{tsai2004visibility}, to estimate both terms in function $S(t)$.
\item In Theorem \ref{thm: value estimates}, we put the estimates together to estimate the function $S(t)$ and obtain an upper and a lower bound for the value of the game under suitable conditions.
The proof of this result only applies to the case when the Lipschitz constant of the curvature function $C_L$ is small enough.
\item Finally, we complete the proof of Theorem \ref{thm: estimates k const} by means of a comparison argument, using an inner and outer ball relative to the obstacle, which allows us to get rid of the restriction on the Lipschitz constant $C_L$.
\end{enumerate}

The smallness condition on $t_0>0$ that will be used from now on reads as
\begin{equation}
\label{t_0 small cond}
\begin{array}{c}
t_0 \leq  \min \left\{ \dfrac{\operatorname{dist} (E, \obs)}{2\gamma_e},
\, \dfrac{\operatorname{dist} (P, \obs)}{2\gamma_p}
\right\}   \\
\noalign{\vspace{6pt}}
\displaystyle\max_{v\in \overline{B(0,1)}} \left| \hls (t_0,v)\right| < L  
  \quad \text{and} \quad
 \displaystyle\max_{v\in \overline{B(0,1)}} \left|\hus (t_0,v)\right| < L
\end{array}
\end{equation}
Roughly, condition \eqref{t_0 small cond} means that the players do not have enough time to neither reach the boundary of the obstacle nor to move their visibility horizons out of the parametrised part of the boundary $\partial\obs_L$.

The following proposition provides a necessary and sufficient condition for the smallness of the value of the game. Moreover, it also provides a representation formula in the case the smallness condition holds.

\begin{proposition}
    \label{prop: value represen formula}
    Consider the two-player surveillance-evasion game in a two-dimensional environment with one obstacle and isotropic and homogeneous dynamics \eqref{homogeneous isotropic intro}. Let the initial position of the game $(E,P)\in \game$ satisfy Assumption \ref{assump: initial pos}, and let $t_0>0$ satisfy \eqref{t_0 small cond}. For any $t\in [0,t_0]$, define the function
    \begin{equation}
    %\label{S def}
    S (t) = \max_{v\in \overline{B(0,1)}} \hus (t,v) - \max_{v\in \overline{B(0,1)}} \hls (t,v).
    \end{equation}
    Here, $\hus (t,v)$ and $\hls (t,v)$ are defined as in \eqref{def horizon limits t,v}.
    
    Then it holds that
    \begin{enumerate}
        \item $V (E,P) \leq t_0$ if and only if $S(t) = 0$ for some $t\in [0,t_0]$.
        \item If $V (E,P) \leq t_0$, then
        $$
        V(E,P) = \min \{ t\geq 0, \quad \text{s.t.} \quad S(t)= 0\}.
        $$
    \end{enumerate}
\end{proposition}

\begin{proof}
    First of all, let us observe that if $\hls (t,v_E) = \hus (t,v_P)$ for some $t\in [0,t_0]$, and $v_E, v_P\in \overline{B(0,1)}$ then the pair of points
    $\left( E + \gamma_e t v_E, \, P + \gamma_p t v_P \right)$ is in $\boundaryT$, which we recall is the target set.
    
    \underline{\textit{Step 1:}} Let us assume that $S(t)=0$ for some $t\in (0,t_0]$. Then, the evader has a strategy with endpoint $E + \gamma_E t v^\ast_E$, for some $v^\ast \in \overline{B(0,1)}$ which satisfies
    $$
    \hus (t,v_P) \leq \hls (t,v^\ast_E) \qquad \forall v_P\in \overline{B(0,1)}.
    $$
    This implies that $V (E,P)\leq t$, since we recall that in the initial position we have $\hus > \hls$. 
    
    \underline{\textit{Step 2:}} For the other direction, let us assume that $V (E,P) = t^\ast\in (0,t_0]$. By the homogeneity of the dynamics in \eqref{homogeneous isotropic intro}, and since the players are assumed to not have enough time to reach the boundary of the obstacle in $[0,t_0]$, we can use Pontryagin's maximum principle to deduce that the optimal trajectory of the game is a straight line in $\free^2$ and that, at the terminal point $\big(E(t^\ast), P(t^\ast)\big)\in \usable$, the direction of the players is orthogonal to $\boundaryT$.
    Namely, there exists a unitary vector $v^\ast$ such that the optimal trajectory of the game $\big(E(\tau), P(\tau)\big)$ is given by
    $$
    E(\tau) = E - \gamma_e v^\ast \tau \quad \text{and} \quad
    E(\tau) = P + \gamma_p v^\ast \tau,
    $$
    where $v^\ast$ is the unit vector normal to the line segment $\Lambda^\ast$ joining $E(t^\ast)$ and $P(t^\ast)$. See Figure \ref{fig: end-game proof} for an illustration.
    
    Now, since $\big(E(t^\ast), P(t^\ast)\big)\in \boundaryT$, the line $\Lambda^\ast$ is also tangent to the obstacle, at the lower-horizon of the visibility from $E(t^\ast)$ and the upper-horizon of the visibility from $P(t^\ast)$. This implies in particular that $v^\ast = \bfn (\hls (t^\ast , -v^\ast)) = \bfn(\hus (t^\ast, v^\ast))$, where $\bfn (\hls (t^\ast))$ denotes the outer normal vector to $\partial \obs_L$ at $\hl (t^\ast,-v^\ast) = \Sigma (\hls (t^\ast, -v^\ast))$. Hence, by Lemma \ref{lem: max horizons} (proved below), we deduce that
    $$
    \max_{v\in \overline{B(0,1)}} \hls (t^\ast,v) = \hls (t^\ast, -v^\ast)
    \quad \text{and} \quad
    \max_{v\in \overline{B(0,1)}} \hus (t^\ast,v) = \hus (t^\ast, v^\ast),
    $$
    and since $\hls (t^\ast, -v^\ast) = \hus (t^\ast, v^\ast)$, we deduce that $S(t^\ast) = 0$. This completes the proof of the first statement of the lemma. By the same argument used in the step 1 of this proof, one can easily deduce that $S(t)>0$ for all $t<t^\ast$, which proves the second statement of the lemma.
\end{proof}

\begin{figure}
    \centering
    \includegraphics[scale=.4]{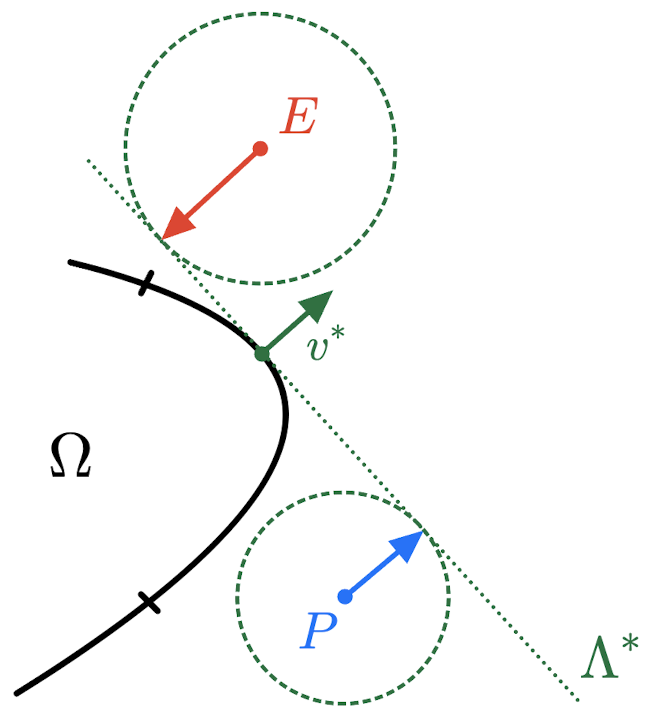}
    \caption{Illustration of the Step 2 in the proof of Proposition \ref{prop: value represen formula}. It represents an end-game situation in which the value is so small that the players cannot reach the boundary of the obstacle. Note that the length of the blue (resp. red) arrow, representing the optimal trajectory of the pursuer (resp. the evader), is proportional to $\gamma_p$ (resp. $\gamma_e$).}
    \label{fig: end-game proof}
\end{figure}

In step 2 of the previous proof, we used the following result (Lemma \ref{lem: max horizons} below), which gives a characterisation of the trajectories maximizing the lower- and upper-horizons for the visibility of $E$ and $P$ respectively, i.e. the vectors maximizing $\hls (t,\cdot)$ and $\hus(t,\cdot)$, when the time $t$ is no greater than $t_0$.
As a by-product, this characterisation is also useful to estimate the maximum of $\hus(t,v)$ and $\hls(t,v)$ in the functions $S(t)$ introduced in Proposition \ref{prop: value represen formula}.
For this purpose, we use the equation for the visibility dynamics from \cite{tsai2004visibility}, 
which for any two smooth trajectories $t\mapsto E(t)$ and $t\mapsto P(t)$ in $\free$, describes the evolution of the visibility horizons $\hl (t)$ and $\hu (t)$ as the solutions to
\begin{equation}
\label{visib dynamics E}
\dot{\mathbf{x}}_E^- (t) = \dfrac{1}{\kappa (\hls (t))} \dfrac{\dot{E}(t) \cdot \bfn (\hls (t))}{|\hl(t) - E(t)|} \bfr (\hls (t))
\end{equation}
and
\begin{equation}
\label{visib dynamics P}
\dot{\mathbf{x}}_P^+ (t) = \dfrac{1}{\kappa (\hus (t))} \dfrac{\dot{P}(t) \cdot \bfn (\hus (t))}{|\hu(t) - P(t)|} \bfr (\hus (t)),
\end{equation}
where $\bfn (\hls (t))$ and $\bfn (\hus (t))$ denote the outer normal vectors to $\partial \obs$ at $\hl(t) = \Sigma (\hls (t))$ and $\hu(t) = \Sigma (\hus(t))$ respectively, $\bfr (\hl (t))$ and $\bfr (\hu (t))$ are the unitary tangent vectors to $\partial\obs$ given by
$$
\bfr (\hls (t)) = \dfrac{\hl (t) - E(t)}{|\hl(t) - E(t)|}
\quad \text{and} \quad
\bfr (\hus (t)) = \dfrac{\hu (t) - P(t)}{|\hu(t) - P(t)|},
$$
and $\kappa (\hls (t))$ and $\kappa (\hus (t))$ denote the curvature of $\partial\obs_L$ at $\hl(t)$ and $\hu(t)$ respectively.

Let us state and prove the following lemma, which gives a characterisation of the trajectories maximizing the visibility horizons for the players.

\begin{lemma}
    \label{lem: max horizons}
    Under the assumptions of Proposition \ref{prop: value represen formula}, for any $t\in (0,t_0]$ there exist a unique $v_P^\ast\in \overline{B(0,1)}$ and a unique $v_E^\ast\in \overline{B(0,1)}$ such that
    $$
    \max_{v\in \overline{B(0,1)}} \hus (t,v) = \hus (t, v_P^\ast)
    \quad \text{and} \quad
    \max_{v\in \overline{B(0,1)}} \hls (t,v) = \hls (t, v_E^\ast).
    $$
    Moreover, it hols that
    \begin{equation}
    \label{opt vectors}
    v_P^\ast = \bfn (\hus (t, v_P^\ast))
    \quad \text{and} \quad
    v_E^\ast = - \bfn (\hls (t, v_E^\ast)),
    \end{equation}
    where
    $\hu (t, v_P^\ast) = \Sigma \left( \hus (t, v_P^\ast) \right)$ and
    $\hl (t, v_E^\ast) = \Sigma \left( \hls (t, v_E^\ast) \right)$ are the upper and lower visibility horizons from 
    $P + t\gamma_p v_P^\ast$ and $E + t\gamma_e v_E^\ast$ respectively, and $\bfn (\hus (t, v_P^\ast))$ and $\bfn (\hls (t, v_E^\ast))$ are the corresponding outer normal vectors to $\partial\obs_L$.
\end{lemma}

\begin{proof}
\underline{\textit{Step 1: Characterisation of $v_P^\ast$.}}
Let us recall that the reachable set in time $t$ from the initial position $P$ is the closed ball $\overline{B(P, \, \gamma_p t)}$. By the convexity of $\partial \obs_L$, there exists a unique line which is tangent to both $\partial \obs_L$ and $\overline{B(P, \, \gamma_p t)}$, and leaves them in the same half-space.
In other words, there exists a unique $v_P^\ast\in \overline{B(0,1)}$ such that the line
$$
\Lambda_P^+ (t,v_P^\ast) := \{x\in \R^2 \ ; \quad (x - x_0) \cdot v_P^\ast =0 \}, \qquad \text{with} \ x_0 = P + \gamma_p t v_P^\ast
$$
satisfies $\hu (t,v_P^\ast)\in \Lambda_P^+$ and $(x - x_0) \cdot v_P^\ast \leq 0$ for all $x\in \partial\obs_L\cup \overline{B(P, \, \gamma_p t)}$. Moreover, since $\Lambda_P^+ (t,v_P^\ast)$ is tangent to $\partial\obs_L$, it holds that $v_P^\ast = \bfn (\hus (t,v_P^\ast))$.
See Figure \ref{fig: opt traj Lemma} (left) for an illustration.

Now, for any $v\in \overline{B(0,1)}$ with $v\neq v_P^\ast$, the line segment joining $P+\gamma_p t v$ and $\hu (t,v)\in \partial \obs_L$ is contained in the half-space $\{ (x - x_0) \cdot v_P^\ast \leq 0 \}$ and does not intersect the obstacle, which implies that $\hu (t,v)$ is visible from $P + \gamma_p t v_P^\ast$, and hence, $\hus (t,v) < \hus (t,v_P^\ast)$.

\underline{\textit{Step 2: Characterisation of $v_E^\ast$.}}
The proof follows a similar argument. In this case, we consider the unique line which is tangent to both $\partial \obs_L$ and $\overline{B(E, \, \gamma_e t)}$, and leaves them in opposite half-spaces.
In other words, there exists a unique $v_E^\ast\in \overline{B(0,1)}$ such that the line
$$
\Lambda_E^- (t,v_P^\ast) := \{x\in \R^2 \ ; \quad (x - x_0) \cdot v_E^\ast =0 \}, \qquad \text{with} \ x_0 = E + \gamma_e t v_E^\ast
$$
satisfies $\hl (t,v_E^\ast)\in \Lambda_E^-$ and 
$$
\begin{cases}
    (x - x_0) \cdot v_E^\ast \leq 0 & \forall x\in \overline{B(E, \, \gamma_e t)} \\
    (x - x_0) \cdot v_E^\ast \geq 0 & \forall x \in \partial\obs_L.
\end{cases}
$$
Moreover, since $\Lambda_E^- (t,v_E^\ast)$ is tangent to $\partial\obs_L$, it holds that $v_P^\ast = -\bfn (\hus (t,v_P^\ast))$.
See Figure \ref{fig: opt traj Lemma} (right) for an illustration.

Now, for any $v\in \overline{B(0,1)}$ with $v\neq v_E^\ast$, the lower-horizon $\hl (t,v)\in \partial \obs_L$ is contained in the half-space $\{ (x - x_0) \cdot v_P^\ast \geq 0 \}$, and therefore, the line segment joining $E+\gamma_e t v$ and $\hl (t,v)\in \partial \obs_L$ crosses the line $\Lambda_E^- (t,v_E^\ast)$. This implies that $\hl (t,v)$ is not visible from $E + \gamma_e t v_E^\ast$, and hence, $\hls (t,v) < \hls (t,v_P^\ast)$.
\end{proof}

\begin{figure}
    \centering
    \includegraphics[scale=.5]{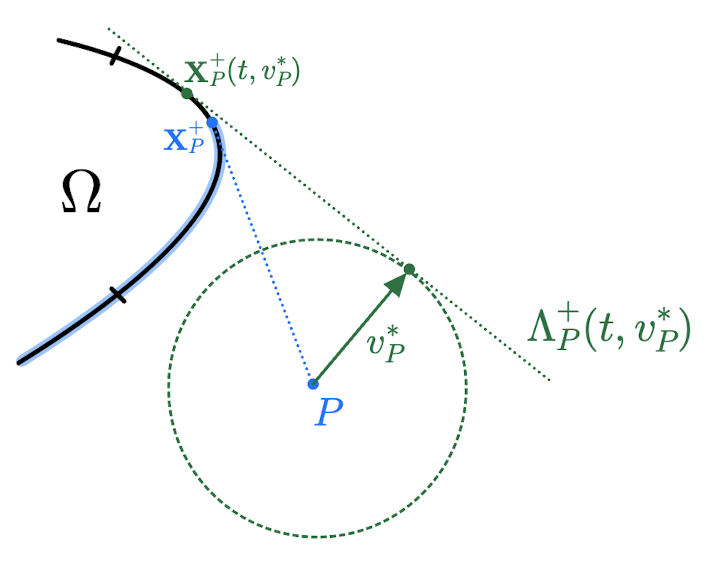}
    \includegraphics[scale=.5]{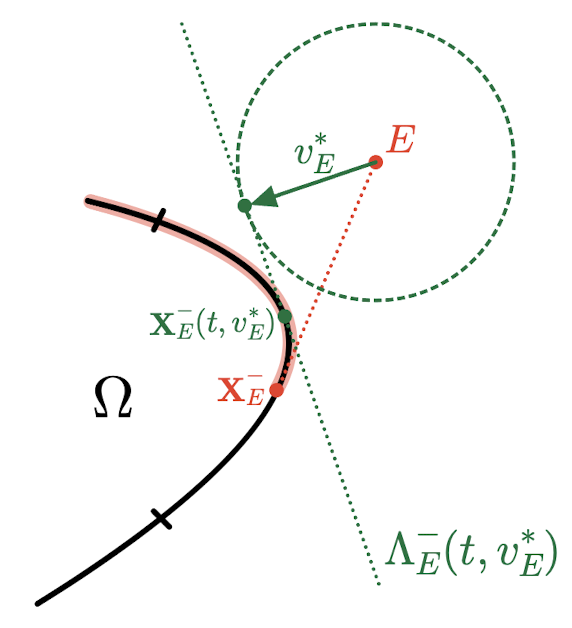}
    \caption{Illustration of the trajectories that maximise the visibility horizon for the pursuer (left) and the evader (right). In the drawing at the left, we see that $v_P^\ast$ has the same direction as the outer normal vector to $\partial\obs_L$ at the horizon point $\hu (t,v_P^\ast)$, whereas in the drawing at the right, we see that $v_E^\ast$ has opposite direction to the outer normal vector to $\partial\obs_L$ at $\hl (t,v_E^\ast)$.}
    \label{fig: opt traj Lemma}
\end{figure}

%\begin{remark}
%\label{rmk: opt traj horizon}
Note in \eqref{opt vectors} that the optimal directions for the pursuer and for the evader have opposite sign with respect to the outer normal vector to the obstacle. This is due to the fact that, while the pursuer is maximizing the upper-horizon for the visibility, the evader is maximizing the lower-horizon.
This implies in particular that the optimal trajectory for the pursuer moves them away from the obstacle, whereas the optimal trajectory for the evader brings them closer to it (see Lemma \ref{lem: estimate dist} in the appendix).
In Lemma \ref{lem: estimate  s bar} in the appendix, we prove that this different behaviour of $E$ and $P$ yields estimates for the maximum of $\hls(t,v)$ and $\hus (t,v)$ with opposite sign in the second-order term,
which plays an important role when analysing the sign of the function $S(t)$ introduced in Proposition \ref{prop: value represen formula}. 
See Figure \ref{fig: opt traj Lemma} for an illustration of the trajectories maximizing the visibility horizons.
%\end{remark}

From Lemma \ref{lem: max horizons}, we deduce that, for any $t\in (0,t_0]$ fixed, the trajectories that maximise the lower and upper visibility horizons at time $t$ for the evader and the pursuer respectively are given by
\begin{equation}
\label{opt traj section}
E(\tau) = E + \tau \gamma_e v_E^\ast
\quad \text{and} \quad
P(\tau) = P + \tau \gamma_p v_P^\ast, \quad \text{for}\ \tau\in [0,t],
\end{equation}
where $v_E^\ast$ and $v_P^\ast$ satisfy \eqref{opt vectors}.
We can now plug these trajectories in equations \eqref{visib dynamics E} and \eqref{visib dynamics P} to estimate the position of the horizons $\hl (t, v_P^\ast)$ and $\hu (t, v_E^\ast)$ defined in Lemma \ref{lem: max horizons}.
This yields the desired estimates for the maximum of $\hls (t,v)$ and $\hus (t,v)$ in \eqref{S def}, which are proved in Lemma \ref{lem: estimate  s bar} in the Appendix \ref{sec: lemmas boundary}.

We now have all the necessary ingredients to prove upper and lower estimates for the value of the game $V (E,P)$ when the initial position $(E,P)\in \game$ satisfies the Assumption \ref{assump: initial pos}.
Before stating the result, let us introduce some quantities which are relevant in the estimates of $V(E,P)$.
Given an initial position of the game $(E,P)\in \game$, we define
\begin{equation}
\label{dist to hor def}
d_E := | E-\hl|
\quad \text{and} \quad
d_P := | P- \hu|,
\end{equation}
representing the distance between the vantage point $E$ (resp. $P$) and the lower (resp. upper) horizon for the visibility from $E$ (resp- $P$).
We also define the quantities
\begin{equation}
\label{d star C def}
d^\ast_\kappa (E,P) = \left(\dfrac{\gamma_p}{\kappa (\hus) d_P} - \dfrac{\gamma_e}{\kappa (\hls) d_E}\right)_+
\quad \text{and} \quad
C(E,P) = \dfrac{\gamma_p^2}{2\kappa (\hus)^2 d_P^3} + \dfrac{\gamma_e^2}{2\kappa (\hls)^2 d_E^3},
\end{equation}
and the quantity
\begin{equation}
\label{S(E,P) def}
S(E,P) = \hus - \hls,
\end{equation}
which represents the distance between the visibility horizons in terms of the parametrisation $\Sigma(\cdot)$. Using the notation of Proposition \ref{prop: value represen formula}, we have $S(E,P) = S(0)$, which is positive whenever $(E,P)$ is in the interior of $\game$ and vanishes on the boundary $\boundaryT$.

Note that the quantity $d^\ast_\kappa (E,P)$ is particularly relevant. Indeed, if $(E,P)\in \boundaryT$, then the visibility horizons coincide, i.e. $\bfx^\ast := \hl = \hu$, which implies that
$$
d^\ast_\kappa (E,P) = \dfrac{1}{\kappa (\mathbf{s}^\ast)} \left( \dfrac{\gamma_p}{|P-\bfx^\ast|} - \dfrac{\gamma_e}{|E-\bfx^\ast|}  \right)_+ , \qquad \text{whenever} \ (E,P)\in \boundaryT,
$$
where $\mathbf{s}^\ast :=\Sigma^{-1} (\bfx^\ast)$.
In view of \eqref{non-usable simple case}, we deduce that $(E,P)\not\in \usable$ whenever $d^\ast_\kappa (E,P)>0$.

\begin{theorem}
    \label{thm: value estimates}
    Consider the two-player surveillance-evasion game in a two-dimensional environment with one obstacle and isotropic and homogeneous dynamics \eqref{homogeneous isotropic intro}. Let the initial position of the game $(E,P)\in \game$ satisfy Assumption \ref{assump: initial pos}, and let $t_0>0$ satisfy \eqref{t_0 small cond}.
    Assume that there exists $\varepsilon\in [0,1)$ such that
    $$
    C_L \leq \dfrac{\varepsilon \kappa_0^3 t_0}{2} \, 
    \min \left\{ \dfrac{\gamma_e}{\kappa (\hls)^2 d_E^2}, \, 
    \dfrac{\gamma_p}{\kappa (\hus)^2 d_P^2}\right\}.
    $$
    
    Then, there exists $0<\overline{t}_0 \leq t_0$, depending on $t_0, \gamma_e,\gamma_p, L, C_L$ and $\kappa_0$, such that
    for any $\delta\in (0,1-\varepsilon)$, it holds that
    \begin{equation}
    \label{lower estimate thm}
    V (E,P) \geq \min \left\{ \dfrac{d^\ast_\kappa (E,P)}{(1+\delta+\varepsilon)C(E,P)} , \, \delta \overline{t}_0 \right\},
    \end{equation}
    and
    \begin{equation}
        \label{upper estimate thm}
        V (E,P) \leq \dfrac{\sqrt{S(E,P)} + d^\ast_\kappa (E,P)}{(1-\varepsilon - \delta) C(E,P)} =: t^\ast,
    \end{equation}
    whenever $0 \leq t^\ast \leq \delta \overline{t}_0$ and $\sqrt{S(E,P)} \leq t^\ast$.
    Let us recall the definition of $S(E,P)$ in \eqref{S(E,P) def}. 
\end{theorem}

\begin{remark}
\begin{enumerate}
    \item We note that in the case where $\partial \obs_L$ has constant curvature (i.e. $C_L=0$ in Assumption \ref{assump: initial pos}), we are allowed to take $\varepsilon=0$. We shall apply this case in the following theorem where, by means of a comparison argument, we are able to get rid of the upper constraint on the Lipschitz constant $C_L$ of the curvature $\kappa(\cdot)$ function.
    \item For the upper estimate to hold we need $\sqrt{S(E,P)}$ to be smaller than the upper estimate. It does not present any inconvenient for our interests, which are mainly focused on estimating the value near the boundary of the game domain $\game$, and we recall the $S(E,P)$ vanishes precisely on $\boundaryT$. 
\item For both estimates \eqref{lower estimate thm} and \eqref{upper estimate thm} to apply,  we need the quantity $d^\ast_\kappa(E,P)$ to be small enough.
Again, this quantity is continuous and vanishes on the usable part of $\boundaryT$.  Hence,  Theorem \ref{thm: value estimates} provides estimates of the value $V(E,P)$ in a neighbourhood of $\usable$
\end{enumerate}
\end{remark}

\begin{proof}
    \underline{\textit{Step 1: Preparation.}}
    Note that the condition on $C_L$ implies that
    \begin{equation}
        \label{C_L small proof}
        \dfrac{C_L \gamma_e}{\kappa_0^3 d_E t_0} \leq \dfrac{\varepsilon \gamma_e^2}{\kappa(\hls)^2 d_E^3}
        \quad \text{and} \quad
        \dfrac{C_L \gamma_p}{\kappa_0^3 d_P t_0} \leq \dfrac{\varepsilon \gamma_p^2}{\kappa(\hus)^2 d_P^3}.
    \end{equation}
    Let $\overline{t}_0\in (0,t_0]$ be sufficiently small, as per Lemma \ref{lem: estimate  s bar} and satisfying
    \begin{equation}
    \label{t0 bar thm}
    \overline{t}_0 \leq \dfrac{C(E,P)}{C_p + C_e},
    \end{equation}
    where $C_p$ and $C_e$ are the constants from Lemma \ref{lem: estimate  s bar}.

    \underline{\textit{Step 2: Lower bound.}} The estimates in Lemma \ref{lem: estimate  s bar} along with \eqref{C_L small proof} yields
    \begin{eqnarray*}
        \max_{v\in \overline{B(0,1)}} \hls (t,v) &\leq &  \hls (0) + \dfrac{\gamma_e}{\kappa (\hls) d_E} t + \dfrac{(1+\varepsilon) \gamma_e^2}{2\kappa (\hls)^2 d_E^3} t^2 + C_e t^3 \\
        \max_{v\in \overline{B(0,1)}} \hus (t,v) &\geq & \hus (0) + \dfrac{\gamma_p}{\kappa (\hus) d_P} t - \dfrac{(1+\varepsilon) \gamma_p^2}{2\kappa(\hus)^2 d_P^3} t^2 - C_p t^3.
    \end{eqnarray*}
    Recalling that the definition of $S(t)$ in Proposition \ref{prop: value represen formula} and the definitions of $d^\ast_\kappa (E,P)$ and $C(E,P)$ in \eqref{d star C def}, we can write
    \begin{eqnarray*}
    S(t) &\geq & S(0) + d^\ast_\kappa (E,P) t - (1+\varepsilon) C(E,P) t^2 - (C_p+C_e) t^3 \\
    &\geq &  d^\ast_\kappa (E,P)t - \big((1+\varepsilon) C(E,P) + t (C_p+C_e) \big)t^2,
    \qquad \forall t\in [0,\overline{t}_0]
    \end{eqnarray*}
    Now, for any $\delta>0$, one can easily verify, using \eqref{t0 bar thm}, that
    $$
    S(t) \geq  d^\ast_\kappa(E,P)t -  (1+\varepsilon + \delta) C(E,P)t^2, \qquad \forall t\in [0,\, \delta \overline{t}_0],
    $$
    which proves that $S(t)>0$ for all $t>0$ satisfying
    $t < \min\left\{ \frac{d^\ast_\kappa (E,P)}{(1 + \varepsilon + \delta) C(E,P)} , \, \delta \overline{t}_0  \right\},$
    and then \eqref{lower estimate thm} follows from Proposition \ref{prop: value represen formula}.

    \underline{\textit{Step 3: Upper bound.}}
    We use again the estimates in Lemma \ref{lem: estimate  s bar} and \eqref{C_L small proof} to obtain
    \begin{eqnarray*}
        \max_{v\in \overline{B(0,1)}} \hls (t,v) \geq \hls (0) + \dfrac{\gamma_e}{\kappa (\hls) d_E} t + \dfrac{(1-\varepsilon) \gamma_e^2}{2\kappa (\hls)^2 d_E^3} t^2 -C_e^3 \\
        \max_{v\in \overline{B(0,1)}} \hus (t,v) \leq \hus (0) + \dfrac{\gamma_p}{\kappa (\hus) d_P} t - \dfrac{(1-\varepsilon) \gamma_p^2}{2\kappa (\hus)^2 d_P^3} t^2 + C_p t^3.
    \end{eqnarray*}
    Using \eqref{t0 bar thm}, and the definitions of $S(E,P)$, $d^\ast_\kappa (E,P)$ and $C(E,P)$, we have
    \begin{eqnarray*}
    S(t) &\leq & S(E,P) + d^\ast_\kappa (E,P) t - (1-\varepsilon) C(E,P) t^2 + (C_p+C_e)t^3 \\
    &\leq & S(E,P) + t \left( d^\ast_\kappa (E,P) - t (1-\varepsilon - \delta) C(E,P)\right), \qquad \forall t\in [0,\delta\overline{t}_0],
    \end{eqnarray*}
    for any $\delta \in (0,1-\varepsilon)$.
    A sufficient condition for $S(t)$ to be non-positive is to have
    $$
     d^\ast_\kappa (E,P) - t (1-\varepsilon - \delta) C(E,P) = - \sqrt{S(E,P)} \quad \text{and} \quad
    t\geq \sqrt{S(E,P)}.
    $$
    Let us set 
    $$
    t^\ast : = \dfrac{\sqrt{S(E,P)} + d^\ast_\kappa (E,P)}{(1-\varepsilon - \delta) C(E,P)}.
    $$
    Then, it follows from Proposition \ref{prop: value represen formula} that $V (E,P) \leq t^\ast$, provided that $t^\ast \leq \delta \overline{t}_0$ and $t^\ast \geq \sqrt{S(E,P)}$.
\end{proof}

Next, we get rid of the upper-bound condition on the Lipschitz constant $C_L$ by means of a comparison argument, using an inner and outer ball with respect to the parametrised part of the boundary in Assumption \ref{assump: initial pos}.
The idea is to consider the same surveillance-evasion game with a smaller (resp. bigger) obstacle $\tilde{\obs}$ with locally constant curvature, noting that, whenever $\tilde{\obs}\subset \obs$ (resp. $\obs\subset \tilde{\obs}$), then $V (E,P) \leq \tilde{V}_\infty (E,P)$ (resp. $\tilde{V}_\infty (E,P) \leq V (E,P)$) for any $(E,P)\in \game $.
In this case, in which we do not use the different value of the curvature at $\hl$ and $\hu$, the quantities of interest, analogous to \eqref{d star C def}, are given by
$$
d^\ast (E,P) = \left(\dfrac{\gamma_p}{d_P} - \dfrac{\gamma_e}{d_E}\right)_+
\quad \text{and} \quad
C^\ast = \dfrac{2d_P^3 d_E^3}{d_E^3\gamma_p^2 + d_P^3 \gamma_e^2}, 
$$
with $d_E$ and $d_P$ defined as in \eqref{dist to hor def}.

\begin{theorem}
    \label{thm: estimates k const}
    Consider the two-player surveillance-evasion game in a two-dimensional environment with one obstacle and isotropic and homogeneous dynamics\eqref{homogeneous isotropic intro}. Let the initial position of the game $(E,P)\in \game$ satisfy Assumption \ref{assump: initial pos}, and let $t_0>0$ satisfy \eqref{t_0 small cond}.
    Then, there exists $0<\overline{t}_0 \leq t_0,$ depending on $t_0,\gamma_e, \gamma_p, L, C_L$ and $\kappa_0$, such that for any $\delta \in (0,1)$, it holds that
    \begin{equation}
        \label{value lower est k const}
        V (E,P) \geq \min \left\{ \dfrac{\kappa_0}{1+\delta} C^\ast d^\ast (E,P) , \, \delta \overline{t}_0 \right\},
    \end{equation}
    and if we take $\delta\geq \underline{\delta}:= 1- (\kappa_0 + LC_L)^2 C^\ast$, it holds that
    \begin{equation}
        \label{value upper est k const}
        V (E,P) \leq \dfrac{\kappa_0 + L C_L}{1-\delta} C^\ast \left( (\kappa_0 + L C_L) \sqrt{S(E,P)} + d^\ast (E,P) \right) =: t^\ast,
    \end{equation}
    whenever $0 \leq t^\ast\leq \delta \overline{t}_0$.
\end{theorem}

Let us recall the definition of $S(E,P)$ in \eqref{S(E,P) def},  which by the Lipschitz continuity of the parametrisation of $\obs_L$ satisfies $S(E,P)\leq C |\hu - \hl|$.
This inequality can be used to obtain the estimate in Theorem \ref{thm: smooth obst intro} from the conclusion of Theorem \ref{thm: estimates k const}.

\begin{proof}
    \underline{\textit{Step 1: Lower bound.}} Let us denote by $\tilde{\obs}$ the open ball of radius $1/\kappa_0$ tangent to $\partial\obs_L$ at $\hu$. Let us denote by $\tilde\game$ the game domain associated to the obstacle $\tilde\obs$, and by $\tilde{V}_\infty (E,P)$ the associated value of the game.     
    By \eqref{hyp smooth and convex}, we have $\partial \obs_L \subset \tilde{\obs}$, and then, one can easily deduce that, for any $(E,P)\in \game\cap \tilde{\game}$ satisfying Assumption \ref{assump: initial pos}, if $V (E,P) \leq t_0$, then $V (E,P)\geq \tilde{V}_\infty (E,P)$. Indeed, by the choice of $t_0$ in \eqref{t_0 small cond}, any game trajectory $(E(t), P(t))\in \game$ terminating at some time $t^\ast \leq t_0$ satisfies that the segment joining $E(t^\ast)$ and $P(t^\ast)$ is tangent to $\partial\obs_L$. Hence, since $\partial\obs_L\subset \tilde{\obs}$, the same trajectory with respect to $\tilde\obs$, has to terminate in a time no larger than $t^\ast$.

    Next, if we denote by $\tilde{\mathbf{x}}_E^-$ and $\tilde{\mathbf{x}}_P^+$ the lower and upper visibility horizons from $E$ and $P$ with respect to $\tilde{\obs}$, and by
    $$
    \tilde{d}_E = | \tilde{\mathbf{x}}_E^- - E|
    \quad \text{and} \quad
    \tilde{d}_P = | \tilde{\mathbf{x}}_P^+ - P|
    $$
    the distances from $E$ and $P$ to such horizons, one can easily verify that $\tilde{d}_P = d_P$ (recall that $\tilde \obs$ is tangent to $\partial \obs_L$ at $\hu$) and $\tilde{d}_E \geq d_E$.

    Now, we can apply Theorem \ref{thm: value estimates} with $C_L=0$, i.e. we can take $\varepsilon = 0$. By plugging the constant curvature version ($\kappa (\mathbf{s}) = \kappa_0$) of $d^\ast_\kappa (E,P)$ and $C(E,P)$ in \eqref{d star C def}, we obtain
    $$
    V (E,P) \geq \min \left\{ \dfrac{\kappa_0}{1+\delta} C^\ast\left( \dfrac{\gamma_p}{\tilde{d}_P} - \dfrac{\gamma_e}{\tilde{d}_E} \right)_+ , \, \delta \overline{t}_0 \right\}.
    $$
    The conclusion then follows from  $\tilde{d}_P = d_P$ and $\tilde{d}_E \geq d_E$, which implies
    $$
    \dfrac{\gamma_p}{\tilde{d}_P} - \dfrac{\gamma_e}{\tilde{d}_E} \geq  \dfrac{\gamma_p}{d_P} - \dfrac{\gamma_e}{d_E}  \geq d^\ast (E,P).
    $$

    \underline{\textit{Step 2: Upper bound.}}
    We use a very similar argument, in which $\tilde\obs$ now denotes the open ball of radius $1/(\kappa_0 + L C_L)$ tangent to $\partial\obs_L$ at $\hu$. In this case, since the curvature of the ball $\kappa_0 + L C_L$ is bigger than $\kappa (\mathbf{s})$ for all $\mathbf{s}\in [-L,L]$, we deduce that $\partial\obs_L \subset \R^2\setminus \tilde\obs$, which in turn implies that $V (E,P) \leq \tilde{V}_\infty (E,P)$ whenever $V (E,P) \leq t_0$.

    From Theorem \ref{thm: value estimates} with constant curvature $C_L = 0$ and $\kappa (\mathbf{s}) = \kappa_0 + LC_L$, we deduce that
    $$
    V (E,P) \leq \dfrac{\kappa_0 + L C_L}{1-\delta} C^\ast \left( (\kappa_0 + L C_L) \sqrt{S(E,P)} + \left(\dfrac{\gamma_p}{\tilde{d}_P} - \dfrac{\gamma_e}{\tilde{d}_E}\right)_+ \right) =: \tilde{t}^\ast,
    $$
    whenever $0 \leq \tilde{t}^\ast\leq \delta \overline{t}_0$ and $\sqrt{S(E,P)} \leq \tilde{t}^\ast$.
    The conclusion follows by checking that $\tilde{d}_P = d_P$ and $\tilde{d}_E \leq d_E$, which implies
    $$
    \left(\dfrac{\gamma_p}{\tilde{d}_P} - \dfrac{\gamma_e}{\tilde{d}_E}\right)_+ \leq  \left(\dfrac{\gamma_p}{d_P} - \dfrac{\gamma_e}{d_E}\right)_+  = d^\ast (E,P).
    $$
    Note also that, since $\tilde{t}^\ast \leq t^\ast$, then $t^\ast \leq \delta \overline{t}_0$ implies $\tilde{t}^\ast\leq \delta \overline{t}_0$.
    Moreover, the choice $\delta\geq \underline{\delta}$ ensures $\sqrt{S(E,P)} \leq \tilde{t}^\ast$.
\end{proof}

We can now use the estimates of Theorem \ref{thm: estimates k const} to obtain estaimates of the profile of the value on the boundary of $\target$ (see Remark \ref{rmk: boundary profile} in Section \ref{sec: main results}).

\begin{corollary}
    \label{cor: boundary estimates smooth}
    Consider the two-player surveillance-evasion game in a two-dimensional environment with one obstacle and homogeneous dynamics. Let $(E,P)\in \boundaryT \setminus \usable$ with $d_E,d_P > \underline{d}>0$, and assume that the segment $[E,P]$ is tangent to $\obs$ on a single point $\bfx^\ast$ satisfying \eqref{assum boundary smooth intro}. 
    
    Then, there exist two constants $d_0 >0$, depending on $\gamma_e,\gamma_p, \kappa_0, \underline{d}, r$ and $\lambda_r$, such that,
    if $d^\ast (E,P)< d_0$, then for any sequence $\{(E_n,P_n)\}_{n\geq 1}\in \game$ converging to $(E,P)$, it holds that
    $$
    \liminf_{n\to \infty} V (E_n, P_n) \geq 
    \dfrac{\kappa_0}{1+\delta} C^\ast d^\ast (E,P)
    $$
    and
    $$
    \limsup_{n\to \infty} V (E_n, P_n) \leq 
    \dfrac{\kappa_0 + r\lambda_r}{1-\delta} C^\ast d^\ast (E,P),
    $$
    where $C^\ast := \dfrac{2d_E^3 d_P^3}{d_E^3\gamma_p^2 + d_P^3 \gamma_e^2}$
    and 
    $\delta = \max \{ \underline{\delta}, 1/2\}$, with $\underline{\delta}:= 1-(\kappa_0+LC_L)^2 C^\ast$.
\end{corollary}

\begin{proof}
    We note that, by the assumption \eqref{assum boundary smooth intro}, and since the sequence $(E_n,P_n)$ converges to $(E,P)$, there exists a sufficiently large $n_0\geq 1$ such that the initial position of the game $(E_n, P_n)\in \game$ satisfies the Assumption \ref{assump: initial pos} for all $n\geq n_0$, with $L>0$ fixed and so small that $\partial\obs_L \subset \overline{B} (\bfx^\ast, r)$. Then, $t_0>0$ satisfying \eqref{t_0 small cond} can be chosen in terms of $\gamma_e,\gamma_p, \underline{d}$ and $r$.
    Note that $\overline{t}_0$ in Theorem \ref{thm: estimates k const} can be chosen independently of $n$.

    For any $n\geq 1$, let us set
    $$
    C_n^\ast = \dfrac{2 d_{P_n}^3 d_{E_n}^3}{d_{E_n}^3 \gamma_p^2 + d_{P_n}^3 \gamma_e^2}. 
    $$
    We can now apply Theorem \ref{thm: estimates k const} to deduce that, for all $n\geq n_0$ and $\delta\in (0,1)$, we have
    \begin{equation}
    \label{ineq1}
    V (E_n,P_n) \geq \dfrac{\kappa_0}{1+\delta} C^\ast_n d^\ast (E_n,P_n)
    \end{equation}
    and
    \begin{equation}
    \label{ineq2}
    V (E_n,P_n) \leq \dfrac{\kappa_0 + LC_L}{1-\delta} C^\ast_n \left( (\kappa_0 + LC_L)\sqrt{S(E_n,P_n)} + d^\ast (E_n,P_n) \right) =: t^\ast_n,
    \end{equation}
    provided $t^\ast_n \leq \delta \overline{t}_0$ and $\delta\geq \underline{\delta}:= 1-(\kappa_0 + L C_L)^2C^\ast$.

    Since $(E_n,P_n)$ converges to $(E,P)\in \boundaryT$, it follows that the difference between the visibility horizons, represented by $S(E_n,P_n)$ converges to $0$ as $n\to \infty$.
    We also have that $d^\ast(E_n,P_n)\to d^\ast(E,P)$ and $C_n^\ast \to C^\ast >0$. We can then choose $\delta^\ast = \max \left\{ \underline{\delta}, 1/2\right\}$,
     and we obtain that for $n\geq n_0$, with $n_0$ large enough, the inequalities \eqref{ineq1} and \eqref{ineq2} hold true provided
    $$
    d^\ast (E,P) <  \dfrac{(1-\delta)\delta \overline{t}_0}{(\kappa_0 + LC_L) C^\ast}.
    $$
    The conclusion the follows by taking the limit as $n\to\infty$ in \eqref{ineq1} and \eqref{ineq2}. Note that the constants $L$ and $C_L$ from Assumption \ref{assump: initial pos} can be controlled by the constants $r$ and $\lambda_r$ from condition \eqref{assum boundary smooth intro}.
    \end{proof}

Next we prove that, when $\kappa_0$ is sufficiently large, we can obtain a sharper estimate of the boundary profile. Namely, we can compute explicitly the first-order term of the profile of the value in the transition between the usable and the nonusable part of $\boundaryT$.
The key is that when $\kappa_0$ is sufficiently big, we are allowed to take any $\delta\in (0,1)$ in Theorem \ref{thm: value estimates}.

\begin{corollary}
    \label{cor: boundary estimates smooth sharp}
    Consider the two-player surveillance-evasion game in a two-dimensional environment with one obstacle and homogeneous dynamics. Let $(E,P)\in \boundaryT \setminus \usable$ with $d_E,d_P > \underline{d}>0$, and assume that the segment $[E,P]$ is tangent to $\obs$ on a single point $\bfx^\ast$ satisfying \eqref{assum boundary smooth intro}. 
    Assume moreover that $\kappa_0^2 \geq 1/C^\ast$, where $C^\ast$ is the same as in Corollary \ref{cor: boundary estimates smooth}.
    
    Then, there exist two constants $d_0, C_0 >0$, depending on $\gamma_e,\gamma_p, \kappa_0, \underline{d}, r$ and $\lambda_r$, such that,
    if $d^\ast (E,P)\leq d_0$, then for any sequence $\{(E_n,P_n)\}_{n\geq 1}\in \game$ converging to $(E,P)$, it holds that
    $$
    \liminf_{n\to \infty} V (E_n, P_n) \geq 
    \kappa_0 C^\ast d^\ast (E,P) - C_0 d^\ast (E,P)^2
    $$
    and
    $$
    \limsup_{n\to \infty} V (E_n, P_n) \leq 
    (\kappa_0 + r\lambda_r) C^\ast d^\ast (E,P) + C_0 d^\ast (E,P)^2.
    $$
\end{corollary}

\begin{proof}
The proof is basically the same as the one for Corollary \ref{cor: boundary estimates smooth} with a different choice of $\delta$. Under the assumption $\kappa_0^2 \geq 1/C^\ast$, we have $\underline{\delta}\leq 0$ in Theorem \ref{thm: estimates k const}, and therefore, the inequalities \eqref{ineq1} and \eqref{ineq2} hold for any choice of $\delta\in (0,1)$ such that $t_n^\ast \leq \delta \overline{t}_0$.
Whence, we can choose $\delta^\ast \in (0,1/2]$ such that
\begin{equation}
\label{choice delta}
d^\ast (E,P) = \dfrac{(1-\delta^\ast)\delta^\ast \overline{t}_0}{2(\kappa_0 + LC_L)C^\ast},
\end{equation}
which ensures that \eqref{ineq1} and \eqref{ineq2} hold for $n$ large enough.
Note that this choice of $\delta^\ast$ can always be done, provided
$$
d^\ast (E,P) < \dfrac{\delta^\ast \overline{t}_0}{8(\kappa_0 + LC_L)C^\ast}.
$$
We conclude the proof by noting that, since $\delta_n^\ast\in (0,1/2]$ we have
$$
\dfrac{1}{1+\delta^\ast} \geq 1- \delta^\ast
\quad \text{and} \quad 
\dfrac{1}{1- \delta^\ast} \leq 1+ 4\delta^\ast.
$$
It only remains to estimate $\delta^\ast$ from above by $C d^\ast (E,P)$, for some constant $C>0$. Using the choice of $\delta^\ast$ in \eqref{choice delta} and $1-\delta_n^\ast \geq 1/2$, we have
$$
\delta^\ast = \dfrac{2(\kappa_0+LC_L)C^\ast}{(1-\delta^\ast)\overline{t}_0} d^\ast (E,P)
\leq \dfrac{4(\kappa_0+LC_L)C^\ast}{\overline{t}_0} d^\ast (E,P).
$$
Note that the constants $L$ and $C_L$ from Assumption \ref{assump: initial pos} can be controlled by the constants $r$ and $\lambda_r$ from condition \eqref{assum boundary smooth intro}.
\end{proof}

%% file: sections/boundary-corner.tex
%%%% Boundary estimates - around a corner

\section{Boundary estimates near a corner}
\label{sec: boundary estimates corner}

Here, in the same framework as in the previous section, we analyse the behaviour of the value function near the usable part of the boundary $\usable$.
Again, we consider the case in which the line segment joining $E$ and $P$ is tangent to $\obs$ at a single point, denoted by $\bfx^\ast\in \partial\obs$.
However, we consider this time the case when the boundary of the obstacle $\partial\obs$ is not differentiable at $\bfx^\ast$, but instead, $\bfx^\ast$ is a corner of the obstacle (see condition \ref{assum horizon corner intro}).
As we can see from Theorems \ref{thm: corner obst intro} and \ref{thm: semipermeable barrier} and Remark \ref{rmk: jump discontinuity corner}, the behaviour differs from the one obtained in the previous section, not only in terms of the boundary profile (how the value transitions from being zero on $\usable$ to being positive on $\boundaryT\setminus \usable$), but also at the level of the structure of the discontinuity set near the interface between the usable and the non-usable part of $\boundaryT$.

The proof of Theorem \ref{thm: corner obst intro} follows a similar idea to that of Theorem \ref{thm: smooth obst intro} in the previous section.
However, one can no longer use the dynamics of the horizon points $\hu (t)$ and $\hl(t)$, since in this case, these are constant and equal to the corner point $\bfx^\ast$. 
Therefore, it is useful to assume, without loss of generality that $\bfx^\ast$ is the origin, and formulate the problem in polar coordinates.

We can write the position of the game $(E, P)$ in polar coordinates as
\begin{equation}
\label{E P polar coord}
E = d_E \left( \cos \theta_E, \sin \theta_E \right)
\qquad \text{and} \qquad
P = d_P \left( \cos \theta_P, \sin \theta_P \right),
\end{equation}
for some $d_E,d_P>0$ and $\theta_E,\theta_P\in \R$.
Note that $d_E$ and $d_P$ represent, precisely, the distance from $E$ and $P$ to the corner point $\bfx^\ast$, assuming that $\bfx^\ast$ is the origin.
In the sequel, we shall consider the following assumption on the initial condition $(E,P)\in \game$.

\begin{assumption}[Initial position of the game]
    \label{assum ini cond corner}
    We assume that there exist $-\frac{\pi}{2} \leq\theta_1 < \theta_2 < \frac{\pi}{2}$ and $r>0$ such that \eqref{assum horizon corner intro} holds.
    We also assume that the initial position of the game $(E,P)\in \game$ is such that the lower-horizon for the visibility from $E$ and the upper-horizon for the visibility from $P$ are both equal to the origin.
    It can be expressed, using polar coordinates as in \eqref{E P polar coord}, by assuming that
    $$
    \theta_1-\dfrac{\pi}{2} < \theta_E < \theta_2 - \dfrac{\pi}{2}
    \quad
    \theta_2-\dfrac{\pi}{2} < \theta_P < \theta_2 + \dfrac{\pi}{2},
    \text{and} \quad
    \theta_E - \theta_P < \pi.
    $$
    We also assume that both players are at positive distance from the corner, i.e. $\exists \underline{d}$ such that $\min \{d_E, d_P\} > \underline{d}$.
    Not that the assumption $\theta_E - \theta_P < \pi$ implies that $(E,P)$ is in the interior of the game domain $\game= \free^2\setminus \target$.
    See Figure \ref{fig:enter-label}.
\end{assumption}

\begin{figure}
    \centering
    \includegraphics[scale=.45]{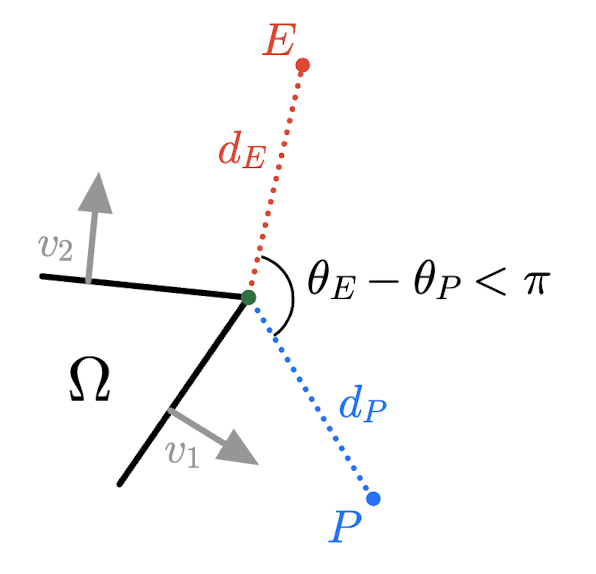}
    \caption{Illustration of the assumption about the initial position of the game described in Assumption \ref{assum ini cond corner}.}
    \label{fig:enter-label}
\end{figure}

Given an initial position of the game $(E,P)\in \game$ satisfying Assumption \ref{assum ini cond corner}, let us recall that, since we assume isotropic and homogeneous dynamics \eqref{homogeneous isotropic intro}, the reachable set in time $t>0$ for $E$ and $P$ are the closed balls $\overline{B(E,\gamma_e t)}$ and $\overline{B(P, \gamma_p t)}$ respectively.
Similarly to the smallness condition for $t_0$ in \eqref{t_0 small cond}, we can assume this time that
\begin{equation}
    \label{t_0 small cond corner}
    \begin{array}{c}
    t_0 \leq  \min \left\{ \dfrac{\operatorname{dist} (E, \obs)}{2\gamma_e},
    \, \dfrac{\operatorname{dist} (P, \obs)}{2\gamma_p}
    \right\}   \\
    \noalign{\vspace{6pt}}
    \hl (t_0,v) = \hu (t_0,v)= \bfx^\ast, \qquad \forall v\in \overline{B(0,1)},
\end{array}
\end{equation}
where $\hl (t_0,v)$ and $\hu (t_0,v)$ denote the upper and lower visibility horizons from $E+\gamma_e t_0 v$ and $P+\gamma_p t_0 v$ respectively.
This condition ensures that the players do not have enough time to reach the obstacle, nor to move their visibility horizons from $\bfx^\ast$.

Under this smallness condition on $t_0$, given a game trajectory $(E(t),P(t))$ in a time interval $(0,t_0]$, it is not difficult to prove that the game ends in the interval $(0,t_0]$ if and only if there exists $\hat{t}\in (0,t_0]$ such that
$$
\theta_E(\hat{t}) - \theta_P (\hat{t}) = \pi,
$$
where $\theta_E(t)$ and $\theta_P(t)$ denote the angular coordinates in $[-\pi,\pi]$ of the trajectories $E(t)$ and $P(t)$ respectively.
Note that, by the assumption \eqref{t_0 small cond corner}, these coordinates are confined in the interval $\left(\theta_1-\frac{\pi}{2}, \theta_2 + \frac{\pi}{2}\right)$, and therefore, there is no ambiguity.  

Next we state the proposition analogous to Proposition \ref{prop: value represen formula}, which provides a sufficient condition for the value of the game, along with a representation formula when this condition is fulfilled.

\begin{proposition}
    \label{prop: represen form corner}
    Consider the two-player surveillance-evasion game in a two dimensional environment with one obstacle and isotropic and homogeneous dynamics \eqref{homogeneous isotropic intro}. Let the initial condition of the game $(E,P)\in \game$ satisfy Assumption \ref{assum ini cond corner}, and let $t_0>0$ satisfy \eqref{t_0 small cond corner}.
    For any $t\in [0,t_0]$, define the function
    $$
    S(t) = \max_{v\in \overline{B(0,1)}} \theta_E (t, v) - \max_{v\in \overline{B(0,1)}} \theta_P (t, v),
    $$
    where $\theta_E (t, v)$ and $\theta_P (t, v)$ denote the angular coordinates in $[-\pi,\pi]$ of $E+\gamma_e t v$ and $P+\gamma_p tv$.

    Then it holds that
    \begin{enumerate}
        \item $V (E,P) \leq t_0$ if and only if $S(t)=\pi$ for some $t\in [0,t_0].$
        \item If $V (E,P)\leq t_0$, then
        $$
        V (E,P) = \min \{ t\geq 0, \quad \text{s.t.} \quad S(t) = \pi\}.
        $$
    \end{enumerate}
\end{proposition}

\begin{proof}
    The proof is very similar to that of Proposition \ref{prop: value represen formula} and is omitted. One only needs to notice that, for $t\leq t_0$, the position of the game $(E(t),P(t))$ reaches the boundary of the game domain $\boundary$ if and only if $\theta_E(t) - \theta_P(t) = \pi$ for some $t\in (0,t_0]$, i.e. if the points $E(t), P(t)$ and $\bfx^\ast$ are aligned.
\end{proof}

In the following lemma,  we provide an analytic expression for the function $S(t)$ in Proposition \ref{prop: represen form corner}. This obviously consists on explicitly computing, for any $t\in (0,t_0]$,  the maximum over  $v\in \overline{B(0,1)}$ of $\theta_E (t,v)$ and $\theta_P(t,v)$. 

\begin{lemma}
\label{lem: S(t) corner explicit}
Consider the two-player surveillance-evasion game in a two-dimensional environment with one obstacle and isotropic and homogeneous dynamics \eqref{homogeneous isotropic intro}. Let the initial condition of the game $(E,P)\in \game$ satisfy Assumption \ref{assum ini cond corner}, and let $t_0>0$ satisfy \eqref{t_0 small cond corner}.
    For any $t\in [0,t_0]$, define the function $S(t)$ from Proposition \ref{prop: represen form corner} is given by
    $$
    S(t) = \theta_E - \theta_P + \arcsin \left( \dfrac{\gamma_e t}{d_E} \right) - \arcsin \left( \dfrac{\gamma_p t}{d_P} \right) \qquad \forall t\in [0, t_0].
    $$
\end{lemma}

\begin{proof}
We need to compute, for any $t\in (0,t_0]$, the maximum over $v\in \overline{B(0,1)}$ of $\theta_E (t,v)$ and $\theta_P(t,v)$. Let us focus on $\theta_P(t,v)$. Since $t\leq t_0$, the reachable set in time $t$ from $P$ is the closed ball $\overline{B(P, \gamma_p t)}$. Then, it is not difficult to prove that there exists a unique $v_P^\ast\in \overline{B(0,1)}$ which maximises $\theta_P (t,v)$. Moreover, $v_P^\ast$ is such that the line passing through $P + \gamma_ptv_P^\ast$ and $\bfx^\ast$ is tangent to $\overline{B(P, \gamma_p t)}$. See Figure \ref{fig:opt_traj_corner} (left) for an illustration.
    Since the points $\bfx^\ast$, $P$ and $P + \gamma_ptv_P^\ast$ form a right triangle, we have
    $$
    \max_{v\in \overline{B(0,1)}} \theta_P (t,v) - \theta_P = \arcsin \left( \dfrac{\gamma_p t}{d_P} \right).
    $$
    Using exactly the same argument, one can deduce that
    $$
    \max_{v\in \overline{B(0,1)}} \theta_E (t,v) - \theta_E = \arcsin \left( \dfrac{\gamma_e t}{d_E} \right).
    $$
    Hence, we can write $S(t)$ from Proposition \ref{prop: represen form corner} as
    $$
    S(t) = \theta_E - \theta_P + \arcsin \left( \dfrac{\gamma_e t}{d_E} \right) - \arcsin \left( \dfrac{\gamma_p t}{d_P} \right) \qquad \forall t\in [0, t_0].
    $$
\end{proof}

\begin{figure}
    \centering
    \includegraphics[scale=.45]{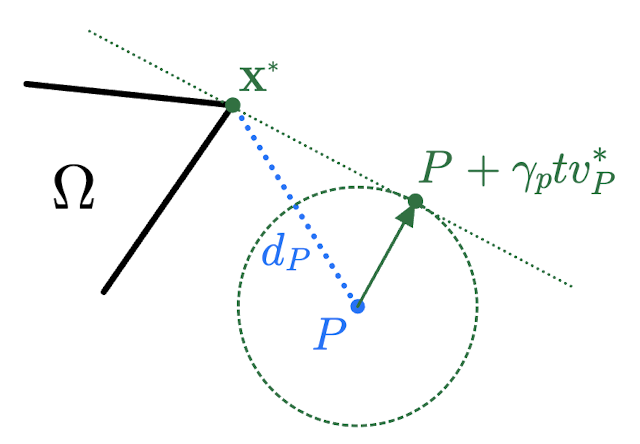}
    \includegraphics[scale=.45]{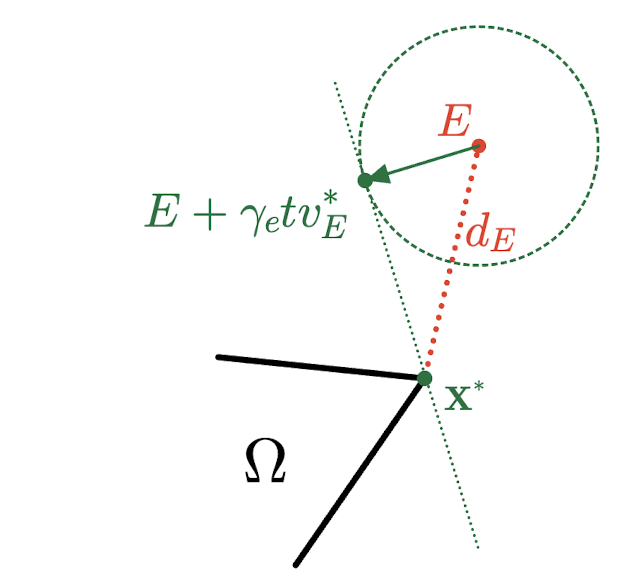}
    \caption{Illustration of the trajectories maximizing the angular coordinate in time $t$ from the initial positions $P$ (left) and $\theta_E(t,v)$ from $E$ (right). Recall that we assumed $\bfx^\ast = 0$.}
    \label{fig:opt_traj_corner}
\end{figure}

We can now proceed with the proof of Theorem \ref{thm: corner obst intro}, which consists in an application of Proposition \ref{prop: represen form corner}, along with the use of the explicit expression of $S(t)$ obtained in Lemma \ref{lem: S(t) corner explicit}.

\begin{proof}[Proof of Theorem \ref{thm: corner obst intro}]
    By virtue of Lemma \ref{lem: S(t) corner explicit}, the function $S(t)$ is given by
    $$
    S(t) = \theta_E - \theta_P + \arcsin \left( \dfrac{\gamma_e t}{d_E} \right) - \arcsin \left( \dfrac{\gamma_p t}{d_P} \right) \qquad \forall t\in [0, t_0].
    $$  
    Since $\arcsin(\cdot)$ is increasing in $[0,1]$, it follows that, whenever $\frac{\gamma_e}{d_E} \leq  \frac{\gamma_p}{d_P}$,
    it holds that $S(t) \leq S(0) < \pi$ for all $\in [0,t_0]$, which by Proposition \ref{prop: represen form corner} implies that $V (E,P)\geq t_0$. 
    
    On the other hand, if $\frac{\gamma_e}{d_E} > \frac{\gamma_p}{d_P}$, one can readily check that the function $t\mapsto S(t)$ is convex, and therefore, we have
    $$
    S(t) \geq \theta_E-\theta_P + \left( \dfrac{\gamma_e}{d_E} - \dfrac{\gamma_p}{d_P} \right) t,\qquad 
    \forall t\in [0,t_0].
    $$
    This implies that, $S(\hat{t})\geq \pi$,  where $\hat{t}:= \dfrac{\pi - (\theta_E-\theta_P)}{\frac{\gamma_e}{d_E} - \frac{\gamma_p}{d_P}} \leq t_0$, and by virtue of Proposition \ref{prop: represen form corner}, $V(E,P)\leq \hat{t}$ provided $\hat{t}\leq t_0$.
\end{proof}

As a corollary of Theorem \ref{thm: corner obst intro}, we can obtain the boundary profile of the value of the game on the boundary of the target set $\target$.

\begin{corollary}
    \label{cor: boundary estimates corner}
    Consider the two-player surveillance-evasion game in a two-dimensional environment with one obstacle and homogeneous dynamics.
    Let $(E,P)\in \boundaryT$ be such that $\min\{d_E,d_P\}>\underline{d}>0$, and assume that the segment $[E,P]$ is tangent to $\obs$ on a single point $\bfx^\ast$ satisfying \eqref{assum horizon corner intro}.
    Then, there exists a constant $t_0>0$, depending on $\gamma_e,\gamma_p,\underline{d}$ and $r$ such that,
    for any sequence $\{(E_n,P_n) \}_{n\geq 1}\subset \game$ converging to $(E,P)$, it holds that
    $$
    \lim_{n\to\infty} V (E_n,P_n) \geq 0, \qquad \text{if $(E,P)\in \usable$,}
    $$
    and
    $$
    \liminf_{n\to\infty} V (E_n,P_n) \geq t_0, \qquad \text{if $(E,P)\in \boundaryT\setminus \usable$.}
    $$
\end{corollary}

\begin{proof}
     Let $\{(E_n,P_n)\}_{n\geq 1}$ be a sequence in $\game$ as in the statement of the Theorem, i.e. $(E_n,P_n)$ converges to a boundary point $(E,P)\in \boundary$, for which the common horizon point $\bfx^\ast$ satisfies \eqref{assum horizon corner intro}. We can assume without loss of generality that $\bfx^\ast$ is the origin.
    By \eqref{assum horizon corner intro}, there exists $n_0\in\N$ such that $(E_n,P_n)$ satisfies Assumption \ref{assum ini cond corner} for all $n\geq n_0$. Since $d_E,d_P>\underline{d}$, we can fix $t_0>0$, independently of $n$, satisfying \eqref{t_0 small cond corner} for all $n\geq n_0$.
    We are therefore in position to apply Proposition \ref{prop: represen form corner}, and the conclusion follows by noting that $\theta_{E_n} - \theta_{P_n} \to \pi$ as $n\to \infty$, and that $t_0$ is independent of $n$.
\end{proof}

We end the section with the proof of Theorem \ref{thm: semipermeable barrier}, which ensures the existence of a semi-permeable barrier emanating from the limit of the usable part of $\boundaryT$.

\begin{proof}[Proof of Theorem \ref{thm: semipermeable barrier}]
In view of the relation $\gamma_e d_P = \gamma_p d_E$, we clearly see that the surface $\mathcal{S}$ is smooth and has dimension three, so it is a smooth hyper-surface.
In order to prove the semi-permeability property,  we write the controls $a(\cdot),b(\cdot): (0,\infty) \to \overline{B(0,1)} $ in polar coordinates as $(a_r (\cdot),  a_\theta (\cdot))$ and $(b_r(\cdot), b_\theta (\cdot))$.
Consider, for each player can use a non-anticipative strategy that mimics the radial component of the opponent's control, i.e.
the strategy for $E$ is given by
$$
\alpha [b_r(\cdot), b_\theta (\cdot)] (t) = b_r(t) \qquad \forall t>0,
$$
and the strategy for $P$ is given by
$$
\beta [a_r(\cdot), a_\theta (\cdot)] (t) = a_r(t) \qquad \forall t>0.
$$
If the initial position $(E,P)$ lies on the hyper-surface $\mathcal{S}$, one can readily verify that by using these non-anticipating strategies the relation $\gamma_e d_P(t) = \gamma_p d_E(t)$ is preserved for all $t>0$.
\end{proof}

%% file: sections/conclusion.tex
%%% Section: preliminaries

\section{Summary and conclusion}
\label{sec: conclusions}

We analyse a class of pursuit-evasion games based on the visibility between the players.
Assuming simple control models for the dynamics of the players, the payoff of the game is given by the time to a target region (the non-visibility region).
The object of interest in this context is the so-called value of the game, which represents the best payoff that each player can ensure, assuming optimal play by the opponent.
The value of the game can be characterised by means of a HJI equation.

Many numerical methods to approximate the value of the game (e.g. Fast Marching and Fast Sweeping methods) rely on the evolution of the solution along the characteristics.
It is therefore of utmost importance to have the correct characterisation of the boundaries in the state space, where the characteristics emanate.

In some parts of the target set, the characteristics propagate towards the interior of the game domain. This is known as the usable part of the boundary.
On the rest of the target set, the value of the game is determined by characteristics arriving from other parts of the game domain. This part of the target set is known as the non-usable part of the boundary, as it cannot be used by the evaders to reach the target set.
We derive a convenient characterisation of the usable/non-usable part of the boundary. 
This will be essential for developing efficient numerical schemes for solving the HJI equation, utilizing the equation's simple structure of characteristics.

Moreover, we analyse the behaviour of the value on the interface between the usable and the non-usable part.
Namely, we describe how the value transitions from being zero on the usable part to being positive on the non-usable part.
Interestingly enough, we show that the value enjoys a different behaviour, which depends on the regularity of the obstacles involved in the game.

Since the original works by Isaacs in the 1960's, the structure of the discontinuity set of the value has attracted a lot of attention among the differential game community.
However, the behaviour near the interface between the usable part and the non-usable part of the boundary is rarely addressed in the literature.
The approach presented in this work might be used to answer similar questions in other differential game problems.

%% file: sections/appendix_2.tex
\section{Auxiliary Lemmas for section \ref{sec: boundary estimates smooth}}
\label{sec: lemmas boundary}

This section is devoted to estimate the maximum over $v\in \overline{B(0,1)}$ of $\hls (t,v)$ and $\hus(t,v)$ appearing in the definition of the function $S(t)$ introduced in Proposition \ref{prop: value represen formula}.
We recall that these functions represent (in terms of the parametrisation $\Sigma(\cdot)$ from Assumption \ref{assump: initial pos}) the maximum attainable for the lower- and upper-horizons of the visibility at time $t>0$, from the initial position $E$ and $P$ respectively.
In view of Lemma \ref{lem: max horizons}, the unique trajectories  attaining such maxima are given by 
\begin{equation}
\label{opt traj appendix}
E(\tau) = E + \tau \gamma_e v^\ast_E
\quad \text{and} \quad
P(\tau) = P + \tau \gamma_p v^\ast_P,
\qquad \text{for} \ \tau \in [0, t],
\end{equation}
where $v^\ast_E$ and $v^\ast_P$ satisfy \eqref{opt vectors}.

The visibility horizons associated to these trajectories can be estimated by using the equations for the visibility dynamics \eqref{visib dynamics E} and \eqref{visib dynamics P} introduced in \cite{tsai2004visibility}. 
A key feature in these equations is the fact that the velocity of the horizon is inversely proportional to the distance between the vantage point and the horizon.
In the case of the trajectories \eqref{opt traj appendix}, 
with $v_E^\ast$ and $v_P^\ast$ as in Lemma \ref{lem: max horizons},
we see that the position of $E(\tau)$ moves closer to the obstacle, whereas $P(\tau)$ moves away from it (see Figure \ref{fig: opt traj Lemma}).
As we shall see in Lemma \ref{lem: estimate  s bar}, this property translates in a different sign for the second-order term in the estimates of the maximum of $\hls (t,v)$ and $\hus(t,v)$, which is crucial when analysing the sign of the function
$S(t) := \max_{v\in \overline{B(0,1)}} \hls (t,v) - \max_{v\in \overline{B(0,1)}} \hus (t,v)$ introduced in Proposition \ref{prop: value represen formula}.

The estimates of $\max_{v\in \overline{B(0,1)}} \hls (t,v)$ and $\max_{v\in \overline{B(0,1)}} \hus (t,v)$ are obtained in two steps:
\begin{enumerate}
    \item In Lemma \ref{lem: estimate dist}, we estimate the distance between the vantage point $E(\tau)$ (resp. $P(\tau)$) and the corresponding horizon point $\hl (\tau)$ (resp. $\hu (\tau)$), associated to the trajectory maximizing $\hls (t)$ (resp. $\hus (t)$).
    \item In Lemma \ref{lem: estimate  s bar}, we plug the estimates from Lemma \ref{lem: estimate dist} in the equations for the visibility \eqref{visib dynamics E} and \eqref{visib dynamics P} to finally obtain the estimates for the maximum of $\hls (t,v)$ and $\hus(t,v)$.
\end{enumerate}

From now on, we denote these distance between the vantage point $E(\tau)$ (resp. $P(\tau)$) and the horizon $\hl (\tau)$ (resp. $\hu(\tau)$) by
\begin{equation}
\label{dist hor def}
d_E (\tau) := |\hl (\tau) - E(\tau)|
\quad \text{and} \quad
d_P (\tau) := |\hu (\tau) - P(\tau)|.
\end{equation}
Similarly, we denote the curvature of $\partial\obs_L$ at the horizon points $\hl(\tau)$ and $\hu(\tau)$ by
\begin{equation}
\label{curv hor def}
\kappa_E (\tau) := \kappa (\hls (\tau))
\quad \text{and} \quad
\kappa_P (\tau) := \kappa (\hus (\tau)).
\end{equation}

\begin{lemma}
    \label{lem: estimate dist}
    Under the assumptions of Proposition \ref{prop: value represen formula}, with $t_0>0$ satisfying \eqref{t_0 small cond}, set $$
    \overline{t}_0 = \min \left\{ t_0 , \,  \frac{\kappa_0^3 t_0}{C_L \| \kappa (\cdot) \|_\infty}, \,
    \dfrac{\kappa_0 t_0}{\sqrt{C_L + \| \kappa (\cdot) \|_\infty^2}}
    \right\}, \quad \big(\text{$\overline{t}_0 = t_0$ in the case $C_L=0$}\big)
    $$
    and for any $t\in (0,\overline{t}_0]$, consider the trajectories \eqref{opt traj appendix}.
    For any $\tau\in [0,t]$, let $\hl (\tau)$ and $\hu (\tau)$ be the lower and upper visibility horizons from $E(\tau)$ and $P(\tau)$ respectively, and let $\hls (\tau) = \Sigma^{-1} (\hl(\tau))$ and $\hus (\tau) = \Sigma^{-1}(\hu (\tau))$.

    Then, there is a constant $C_e>0$ such that
    \begin{eqnarray*}
    d_E(\tau)^2 &\leq & d_E(0)^2
    - \dfrac{2 \gamma_e}{\kappa_E (0)} \tau + C_e t \tau \\
    d_E(\tau)^2 &\geq & d_E(0)^2
    - \dfrac{2 \gamma_e}{\kappa_E (0)} \tau - C_e t \tau,
    \end{eqnarray*}
    for all $\tau\in [0,t]$, and there is a constant $C_p>0$ such that
     \begin{eqnarray*}
    d_P(\tau)^2 &\leq & d_P(0)^2
    + \dfrac{2 \gamma_p}{\kappa_P (0)} \tau + C_p t \tau \\
    d_P(\tau)^2 &\geq & d_P(0)^2
    + \dfrac{2 \gamma_p}{\kappa_P (0)} \tau - C_p t \tau,
    \end{eqnarray*}
    for all $\tau\in [0,t]$. The constants $C_e$ and $C_p$ depend only on $\gamma_e,\gamma_p$, $L$, $C_L$, $\kappa_0$ from the Assumption \ref{assump: initial pos} and $t_0$ from \eqref{t_0 small cond}.
\end{lemma}

\begin{proof}
    Let us start by estimating $d_E(\tau)^2 = |\hl (\tau) - E(\tau)|^2$.
    Using the dynamics of the players \eqref{homogeneous isotropic intro} and of the visibility horizon \eqref{visib dynamics E}, we obtain
    \begin{eqnarray}
        \dfrac{d}{d\tau} | \hl (\tau) -  E(\tau) |^2 &=&
        2 \left( \hl(\tau) - E(\tau) \right) \cdot \left( \dot{\mathbf{x}}_E^- (\tau) - \dot{E} (\tau) \right) \nonumber \\
        &=& \dfrac{2\gamma_e}{\kappa_E (\tau)} v_E^\ast \cdot \bfn (\hls (\tau)) - 2 \gamma_e \left( \hl (\tau) - E(\tau) \right) \cdot v_E^\ast \nonumber \\
        &=& 2\gamma_e A_1(\tau) A_2(\tau) - 2\gamma_e A_3(\tau), \label{dist deriv estimate}
    \end{eqnarray}
    where we recall that $\bfn ( \hls(\tau))$ is the outer normal vector to $\partial\obs_L$ at $\hl (\tau) = \Sigma(\hls(\tau))$.

    The estimates are obtained by controlling the three following quantities:
    \begin{equation}
    \label{A1 A2 A3}
    A_1(\tau) = \frac{1}{\kappa (\hls (\tau))},
    \qquad
    A_2(\tau) = v^\ast_E \cdot \bfn (\hls (\tau))
    \quad \text{and} \quad
    A_3 (\tau) = \left( \hl (\tau) - E(\tau) \right) \cdot v_E^\ast.
    \end{equation}
    
    \underline{\textit{Step 1: Estimate $A_1(\tau)$.}}
    For the first term, we can use the Lipschitz continuity of $\kappa(\cdot)$ in \eqref{hyp smooth and convex} to obtain
    $$
    \dfrac{1}{\kappa (\hls (0))} - \dfrac{C_L}{\kappa_0^2} | \hls (\tau) - \hls (0) | \leq A_1(\tau) \leq 
    \dfrac{1}{\kappa (\hls (0))} + \dfrac{C_L}{\kappa_0^2} | \hls (\tau) - \hls (0) |.
    $$
    Since $\hl (\tau)=\Sigma (\hls (\tau))$, we deduce that $|\dot{\textbf{x}}_E^- (\tau)| = |\Sigma'(\hls (\tau))| |\dot{\mathbf{s}}_E^- (\tau)|$, and therefore, using \eqref{visib dynamics E}, we obtain
    \begin{equation}
    \label{s dot estimate}
    |\dot{\mathbf{s}}_E^- (\tau)| \leq \dfrac{\gamma_e}{\kappa_0 |\hl(\tau) - E(\tau)|} \leq \dfrac{1}{\kappa_0 t_0},
    \end{equation}
    where we used the fact that \eqref{t_0 small cond} implies
    $$
    |\hl (\tau) - E(\tau)| \geq \operatorname{dist}(E(\tau), \obs) \geq \operatorname{dist}(E(0), \obs) - t_0 \gamma_e \geq t_0 \gamma_e.
    $$
    Hence, we have $| \hls (\tau) - \hls (0) | \leq \frac{t}{\kappa_0 t_0}$, which implies
    \begin{equation}
    \label{A1 estimate}
    0 \leq \dfrac{1}{\kappa_E (0)} - \dfrac{C_L}{\kappa_0^3 t_0} t \leq A_1(\tau) \leq 
    \dfrac{1}{\kappa_E (0)} + \dfrac{C_L}{\kappa_0^3 t_0} t,
    \end{equation}
    where the left-most inequality follows from $0< t\leq \overline{t}_0$.

    \underline{\textit{Step 2: Estimate $A_2(\tau)$.}}
    For the second term in \eqref{A1 A2 A3}, we use the assumption \eqref{opt vectors} on $v_E^\ast$ to obtain
    \begin{equation}
    \label{A2 estimate 1}
    A_2(\tau) = \underbrace{v_E^\ast \cdot \bfn (\hls (t))}_{=-1}
    + \, v_E^\ast \cdot \left( \bfn (\hls (\tau)) - \bfn (\hls (t)) \right).
    \end{equation}
    Let us define the function
    $$
    \phi (s) = v_E^\ast \cdot \left( \bfn (s) - \bfn (\hls (t)) \right), \qquad \forall s\in [-L,L]-
    $$
    Using the Frenet-Serret formulas\footnote{If we denote by $\mathbf{t} (s)$ and $\bfn (s)$ the tangent and normal vector to the curve, then $\dot{\mathbf{t}} = \kappa (s) \bfn (s)$ and $\dot{\bfn} (s) = -\kappa (s) \mathbf{t}(s)$.} for the smooth curve $\Sigma (\cdot)$, we can compute
    $$
    \phi'(s) = -\kappa (s) v_E^\ast \cdot \mathbf{t}(s)
    \quad \text{and} \quad
    \phi''(s) = -\kappa'(s) v_E^\ast \cdot \mathbf{t}(s) - \kappa (s)^2 v_E^\ast \cdot \bfn (s). 
    $$
    The assumption \eqref{opt vectors} on $v_E^\ast$ implies $\phi'(\hls (t)) = 0$, and \eqref{hyp smooth and convex} imlpies
    $$
    |\phi''(s)| \leq C_L + \| \kappa(\cdot)\|^2_\infty \qquad \forall s\in [-L,L].
    $$
    We can combine this with \eqref{s dot estimate} to obtain
    $$
    \left|v_E^\ast \cdot \left( \bfn(\hls (\tau)) - \bfn (\hls(t)) \right)\right| \leq  \dfrac{C_L + \| \kappa (\cdot)\|_\infty^2}{2} \left( \hls (\tau) - \hls (t) \right)^2 \leq   \dfrac{C_L + \| \kappa (\cdot)\|_\infty^2}{2} \dfrac{t^2}{\kappa_0^2 t_0^2}.
    $$
    Therefore, it follows from \eqref{A2 estimate 1} that
    \begin{equation}
    \label{A2 estimate 2}
        -1 - \dfrac{C_L + \| \kappa (\cdot)\|_\infty^2}{2\kappa_0^2 t_0^2} t^2
    \leq A_2 (\tau) \leq
    -1 + \dfrac{C_L + \| \kappa (\cdot)\|_\infty^2}{2\kappa_0^2 t_0^2} t^2\leq 0,
    \end{equation}
    where the right-most inequality follows from $0< t\leq \overline{t}_0$.

    \underline{\textit{Step 3: Estimate $A_3(\tau)$.}}
    For the third term in \eqref{A1 A2 A3}, let us define $\mathbf{t}_E (\tau) = \hl (\tau) - E(\tau)$ for all $\tau \in [0,t]$.
    We note that the vector $\mathbf{t}_E(t)$ is tangent to $\partial\obs_L$ at $\hl (t)$, and thus perpendicular to $\bfn (\hl (t))$. Therefore, the hypothesis \eqref{opt vectors} on $v_E^\ast$ implies that
    $$
    -\left| \mathbf{t}_E(t) - \mathbf{t}_E(\tau) \right| \leq A_3 (\tau) \leq \left| \mathbf{t}_E(t) - \mathbf{t}_E(\tau) \right|.
    $$
    We can then use $|\dot{\mathbf{x}}_E^- (\tau)| \leq \frac{1}{\kappa_0 t_0}$ from \eqref{s dot estimate} and $|\dot{E}(\tau)| = \gamma_e$ to obtain
    $$
    \left| \mathbf{t}_E(t) - \mathbf{t}_E(\tau) \right| \leq 
    \left| \hl (t) - \hl(\tau) \right| + \left| E (t) - E(\tau) \right| \leq \left( \dfrac{1}{\kappa_0 t_0} + \gamma_e \right) t,
    $$
    which then yields
    \begin{equation}
        \label{A3 estimate}
        - \left( \dfrac{1}{\kappa_0 t_0} + \gamma_e \right) t \leq A_3 (\tau) \leq  \left( \dfrac{1}{\kappa_0 t_0} + \gamma_e \right) t.
    \end{equation}

    \underline{ \textit{Step 4: Conclusion.}}
    Combining \eqref{dist deriv estimate} and the estimates \eqref{A1 estimate},\eqref{A2 estimate 2},\eqref{A3 estimate}, we obtain
    \begin{eqnarray*}
        \dfrac{d}{d\tau} |\hl (\tau) - E(\tau)|^2 &\leq &
        \left( \dfrac{2\gamma_e}{\kappa_E (0)} - \dfrac{2\gamma_e C_L}{\kappa_0^3 t_0} t \right) \left( -1 + \dfrac{C_L + \| \kappa (\cdot)\|_\infty^2}{2\kappa_0^2 t_0^2} t^2 \right) + \left( \dfrac{2\gamma_e}{\kappa_0 t_0} + 2\gamma_e^2\right) t \\
        &\leq & -\dfrac{2\gamma_e}{\kappa_E (0)} + C_e t,
    \end{eqnarray*}
    for some $C_e>0$ depending on $\gamma_e,\kappa_0, C_L$ and $t_0$. Similarly, we have
    \begin{eqnarray*}
        \dfrac{d}{d\tau} |\hl (\tau) - E(\tau)|^2 &\geq &
        \left( \dfrac{2\gamma_e}{\kappa_E (0)} + \dfrac{2\gamma_e C_L}{\kappa_0^3 t_0} t \right) \left( -1 - \dfrac{C_L + \| \kappa (\cdot)\|_\infty^2}{2\kappa_0^2 t_0^2} t^2 \right) - \left( \dfrac{2\gamma_e}{\kappa_0 t_0} + 2\gamma_e^2\right) t \\
        &\geq & -\dfrac{2\gamma_e}{\kappa_E (0)} - C_e t,
    \end{eqnarray*}
    The upper and lower estimates for $|\hl (\tau) - E(\tau)|^2$ in the statement of the Lemma then follow.
    The estimates for $|\hu (\tau) - P(\tau)|^2$ can be proven similarly. The main difference is that in the estimate of $A_2(\tau)$ in \eqref{A2 estimate 1}, one has $v_P^\ast \cdot \bfn (\hu (t)) = 1$, which in turn implies that $A_2(\tau)\geq 0$ for all $\tau\in[0,t]$, provided $t\leq \overline{t}_0$.
\end{proof}

We are now in position to estimate the maximum over $v\in \overline{B(0,1)}$ of $\hus (t,v)$ and $\hls (t,v)$ appearing in the function $S(t)$ introduced in Proposition \ref{prop: value represen formula}.

\begin{lemma}
\label{lem: estimate  s bar}
Under the assumptions of Proposition \ref{prop: value represen formula}, with $t_0 >0$ satisfying \eqref{t_0 small cond}.
Then there exist $C_e, C_p>0$ and a small enough $0< \overline{t}_0 \leq t_0$, depending on $t_0, \gamma_e,\gamma_p, L,C_L$ and $\kappa_0$, such that
\begin{eqnarray*}
    \max_{v\in \overline{B(0,1)}} \hus (t,v) (t) &\leq & \hls (0) + \dfrac{\gamma_e}{\kappa_E (0) d_E(0)} t + \left( \dfrac{\gamma_e^2}{2\kappa_E (0)^2 d_E(0)^3} + \dfrac{\gamma_e C_L}{\kappa_0^3 d_E(0) t_0} \right) t^2 + C_e t^3 \\
    \max_{v\in \overline{B(0,1)}} \hus (t,v) &\geq & \hls (0) + \dfrac{\gamma_e}{\kappa_E (0) d_E(0)} t + \left( \dfrac{\gamma_e^2}{2\kappa_E (0)^2 d_E(0)^3} - \dfrac{\gamma_e C_L}{\kappa_0^3 d_E(0) t_0} \right) t^2 - C_e t^3
\end{eqnarray*}
and
\begin{eqnarray*}
    \max_{v\in \overline{B(0,1)}} \hus (t,v) &\leq & \hus (0) + \dfrac{\gamma_p}{\kappa_P (0) d_P(0)} t - \left( \dfrac{\gamma_p^2}{2\kappa_P (0)^2 d_P(0)^3} - \dfrac{\gamma_p C_L}{\kappa_0^3  d_P(0) t_0} \right) t^2 + C_p t^3 \\
    \max_{v\in \overline{B(0,1)}} \hus (t,v) &\geq & \hus (0) + \dfrac{\gamma_p}{\kappa_P (0) d_P(0)} t - \left( \dfrac{\gamma_p^2}{2\kappa_P (0)^2 d_P(0)^3} + \dfrac{\gamma_p C_L}{\kappa_0^3 d_P(0) t_0} \right) t^2 - C_p t^3
\end{eqnarray*}
for all $t\in [0, \overline{t}_0]$.
Let us recall that $d_E(0)$ and $d_P(0)$ are defined in \eqref{dist hor def}, and $\kappa_E(0)$ and $\kappa_P(0)$ are defined in \eqref{curv hor def}.
\end{lemma}

\begin{proof}
    We recall that, by virtue of Lemma \ref{lem: max horizons}, the trajectories that maximise  $\hls (t)$ and $\hus (t)$ are
    $$
    E(\tau) = E + \tau \gamma_e v^\ast_E
    \quad \text{and} \quad
    P(\tau) = P + \tau \gamma_p v^\ast_P,
    \qquad \text{for} \ \tau \in [0, t].
    $$
    Hence, we can use \eqref{visib dynamics E} and \eqref{visib dynamics P} to deduce that $\max_{v\in \overline{B(0,1)}} \hls (t,v) = \hls(t)$ and $\max_{v\in \overline{B(0,1)}} \hus (t,v) = \hus (t)$, where $\hls(\tau)$ and $\hus(\tau)$ are the solutions to the differential equations
    $$
    \dot{\mathbf{s}}_E^- (\tau) = \gamma_e \dfrac{|v_E^\ast \cdot \bfn (\hls(\tau))|}{\kappa_E(\tau) d_E(\tau)}  
    \quad \text{and} \quad
    \dot{\mathbf{s}}_P^+ (\tau) = \gamma_p \dfrac{|v_P^\ast \cdot \bfn (\hus(\tau))|}{\kappa_P(\tau) d_P(\tau)} ,
    \quad \text{for} \ \tau \in (0,t),
    $$
    where $d_E(\tau)$ and $d_P(\tau)$ are defined in \eqref{dist hor def} and estimated in Lemma \ref{lem: estimate dist}.

    \underline{\textit{Step 1: Estimates for $\hls (t)$.}}
    Using the notation introduced in \eqref{A1 A2 A3}, we can write
    \begin{equation}
    \label{s_E^+ equation} 
    \dot{\mathbf{s}}_E^- (\tau) = \gamma_e A_1 (\tau) |A_2 (\tau)| \dfrac{1}{d_E(\tau)}.
    \end{equation}
    From \eqref{A1 estimate} and \eqref{A2 estimate 2}, it follows that
    $$
    \gamma_e A_1(\tau) |A_2(\tau)| \geq \left( \dfrac{\gamma_e}{\kappa_E (0)} - \dfrac{\gamma_e C_L}{\kappa_0^3 t_0}t \right) \left( 1 - \dfrac{C_L + \| \kappa(\cdot)\|_\infty^2}{2 \kappa_0^2 t_0^2} t^2\right) $$
    and
    $$
    \gamma_e A_1(\tau) |A_2(\tau)| \leq
        \left( \dfrac{\gamma_e}{\kappa_E (0)} + \dfrac{\gamma_e C_L}{\kappa_0^3 t_0}t \right) \left( 1+ \dfrac{C_L + \| \kappa(\cdot)\|_\infty^2}{2 \kappa_0^2 t_0^2} t^2\right),
    $$
    which in turn implies
    \begin{equation}
        \label{A1A2 estimate}
        \dfrac{\gamma_e}{\kappa_E (0)} - \dfrac{\gamma_e C_L}{\kappa_0^3 t_0}t - C_1 t^2 
        \leq \gamma_e A_1(\tau) |A_2(\tau)| \leq 
        \dfrac{\gamma_e}{\kappa_E (0)} + \dfrac{\gamma_e C_L}{\kappa_0^3 t_0}t + C_1 t^2,
    \end{equation}
    for some constant $C_1>0$.

    We can now use the upper and lower estimates of $d_E(\tau)$ from Lemma \ref{lem: estimate dist} to obtain
    $$
    \dfrac{1}{\sqrt{d_E(0)^2 - C_t' \tau}} \leq \dfrac{1}{d_E(\tau)}
    \leq \dfrac{1}{\sqrt{d_E(0)^2 - C_t \tau}}
    $$
    where
    $$
    C_t = \left( \dfrac{2\gamma_e}{\kappa_E (0)} + C_e t \right)
    \quad \text{and} \quad
    C_t' = \left( \dfrac{2\gamma_e}{\kappa_E (0)} - C_e t \right).
    $$
    Note that we can always take $\overline{t}_0\in (0,t_0]$ small enough, so that $2C_t \tau \leq d_E(0)^2$ for all $t\in [0,\overline{t}_0]$ and $\tau \in [0,t]$.
    The choice of $\overline{t}_0$ depends on $d_E(0)$, however, by \eqref{t_0 small cond}, it can be made depending on $t_0$ instead. Namely, it is sufficient to take $\overline{t}_0>0$ such that $2C_t \tau \leq (2\gamma_e t_0)^2$ for all $t\in [0,\overline{t}_0]$ and $\tau \in [0,t]$.
    
    Since $\sqrt{d_E(0)^2 - C_t\tau}$ is uniformly positive for $\tau\in [0,t]$, it is not difficult to verify that there exists a constant $C_2>0$ such that
    \begin{equation}
    \label{1/dE estimate}
    \dfrac{1}{d_E(0)} + \dfrac{C_t'}{2 d_E(0)^3} \tau - C_2 \tau^2 
    \leq \dfrac{1}{d_E(\tau)} \leq 
    \dfrac{1}{d_E(0)} + \dfrac{C_t}{2 d_E(0)^3} \tau + C_2 \tau^2.
    \end{equation}

    \underline{\textit{Step 1.1: Lower estimate.}}
    In view of the lower estimates in \eqref{A1A2 estimate} and \eqref{1/dE estimate}, we can use the equation \eqref{s_E^+ equation} to obtain
    $$
    \dot{\mathbf{s}}_E^- (\tau) \geq
    \left( \dfrac{\gamma_e}{\kappa_E (0)} - \dfrac{\gamma_e C_L}{\kappa_0^3 t_0}t - C_1 t^2  \right) 
    \left( \dfrac{1}{d_E(0)} + \dfrac{C_t'}{2 d_E(0)^3} \tau - C_2 \tau^2 \right) \qquad \forall \tau \in (0,t).
    $$
    Hence, for any $t\in [0,\overline{t}_0]$, we have
    \begin{eqnarray*}
    \max_{v\in \overline{B(0,1)}} \hus (t,v)  &\geq & \hls (0) + \left( \dfrac{\gamma_e}{\kappa_E (0)} - \dfrac{\gamma_e C_L}{\kappa_0^3 t_0}t - C_1 t^2  \right) 
    \left( \dfrac{t}{d_E(0)} + \dfrac{C_t'}{4 d_E(0)^3} t^2 - \dfrac{C_2}{3} t^3 \right) \\
    &\geq & \hls (0) + \dfrac{\gamma_e}{\kappa_E (0) d_E(0)} t 
    + \left( \dfrac{\gamma_e^2}{2 \kappa_E (0)^2 d_E(0)^3} - \dfrac{\gamma_e C_L}{\kappa_0^3 d_E(0) t_0}  \right) t^2 - C_e t^3,
    \end{eqnarray*}
    for some $C_e>0$ depending on $\overline{t}_0,t_0, \gamma_e,C_L$ and $\kappa_0$. Here we used the definition of $C_t'$, and put all the terms of order higher or equal than 3 in the constant $C_e$.

    \underline{\textit{Step 1.2: Upper estimate.}}
    Similarly, we can now use the upper estimates in \eqref{A1A2 estimate} and \eqref{1/dE estimate}, we can use the equation \eqref{s_E^+ equation} to obtain, for all $t\in [0,\overline{t}_0]$,
    \begin{eqnarray*}
    \max_{v\in \overline{B(0,1)}} \hus (t,v)  &\leq & \hls (0) +   \left( \dfrac{\gamma_e}{\kappa_E (0)} + \dfrac{\gamma_e C_L}{\kappa_0^3 t_0}t + C_1 t^2  \right)
    \left( \dfrac{t}{d_E(0)} + \dfrac{C_t}{4 d_E(0)^3} t^2 + \dfrac{C_2}{3} t^3 \right) \\
    &\leq & \hls (0) + \dfrac{\gamma_e}{\kappa_E (0) d_E(0)} t
    + \left( \dfrac{\gamma_e^2}{2 \kappa_E (0)^2 d_E(0)^3} + \dfrac{\gamma_e C_L}{\kappa_0^3 d_E(0) t_0}  \right) t^2 + C_e t^3.
    \end{eqnarray*}

    \underline{\textit{Step 2: Estimates for $\hus(t)$.}}
    Using the analogous notation as in \eqref{A1 A2 A3}, we can write
    $$
    \dot{\mathbf{s}}_P^+ (\tau) = \gamma_p A_1 (\tau) |A_2 (\tau)| \dfrac{1}{d_P(\tau)},
    $$
    and by the same argument used to obtain \eqref{A1A2 estimate}, one can obtain
    $$
        \dfrac{\gamma_p}{\kappa_P (0)} - \dfrac{\gamma_p C_L}{\kappa_0^3 t_0}t - C_1 t^2 
        \leq \gamma_p A_1(\tau) |A_2(\tau)| \leq 
        \dfrac{\gamma_p}{\kappa_P (0)} + \dfrac{\gamma_p C_L}{\kappa_0^3 t_0}t + C_1 t^2,
    $$
    for some $C_1>0$.

    We can now use the upper and lower estimates of $d_P(\tau)$ in Lemma \ref{lem: estimate dist} to obtain
    $$
    \dfrac{1}{\sqrt{d_P(0)^2 + C_t' \tau}} \leq \dfrac{1}{d_P(\tau)}
    \leq \dfrac{1}{\sqrt{d_P(0)^2 + C_t \tau}}
    $$
    where
    $$
    C_t = \left( \dfrac{2\gamma_p}{\kappa_P (0)} - C_e t \right)
    \quad \text{and} \quad
    C_t' = \left( \dfrac{2\gamma_p}{\kappa_P (0)} + C_p t \right).
    $$

    Again, it is not difficult to verify that there exists a constant $C_2>0$ such that
    $$
    \dfrac{1}{d_P(0)} - \dfrac{C_t'}{2 d_P(0)^3} \tau - C_2 \tau^2 
    \leq \dfrac{1}{d_P(\tau)} \leq 
    \dfrac{1}{d_P(0)} - \dfrac{C_t}{2 d_P(0)^3} \tau + C_2 \tau^2.
    $$
    The upper and lower estimates for $\hus (t)$ follow from a similar computation as for $\hls(t)$. 
\end{proof}